%% file: paper_siam.tex
\documentclass{siamart0216}
\usepackage{latexsym}
\usepackage{amssymb}
\usepackage{amsfonts}
\usepackage{amsmath}
\usepackage{mathrsfs}
\usepackage{color,xcolor}
\usepackage{ifpdf}
\usepackage{psfrag}
\usepackage{graphicx,graphics,subfigure,epsfig}
\usepackage{url}
\usepackage{hyperref}
\usepackage{xspace}
\usepackage{nicefrac}

\newcommand{\dpar}[2]{\dfrac{\partial #1}{\partial #2}}

 \newcommand{\R}{\mathbb R}
 \newcommand{\Z}{\mathbb Z}

\renewcommand{\P}{\mathbb P}

\newcommand{\QQ}{\mathcal Q}

\newcommand{\V}{\mathcal{V}}
\newcommand{\E}{\mathcal{E}}

\newcommand{\bbf}{{\mathbf {f}}}
\newcommand{\bn}{{\mathbf {n}}}
\newcommand{\bm}{\mathbf{m}}

\newcommand{\bu}{\mathbf{u}}

\newcommand{\bV}{\mathbf{V}}

\newcommand{\hbbf}{\hat{\mathbf{f}}}

\newcommand{\bx}{\mathbf{x}}

\newcommand{\norm}[1]{\big |\big | #1\big |\big |}

\newtheorem{thm}{Theorem}[section]

\newtheorem{prop}[thm]{Proposition}
\newcommand{\remi}[1]{{#1}}
\graphicspath{{arxiv/}}
\begin{document}
	\title{Staggered  schemes for compressible flow: a general construction}
\date{}
\author{R. Abgrall\thanks{ Institute of Mathematics, University of Z\"urich,  Z\"urich, Switzerland. \email{remi.abgrall@math.uzh.ch}}
}
\maketitle
\begin{abstract}
This paper is focused on the approximation of the Euler equations of compressible fluid dynamics on a staggered mesh. With this aim, the flow parameters are described by the velocity, the density and the internal energy. The thermodynamic quantities are described on the elements of the mesh, and thus the approximation is only in $L^2$, while the kinematic quantities are globally continuous. The method is general in the sense that the thermodynamic and kinetic parameters are described by an arbitrary degree of polynomials. In practice, the difference between the degrees of the kinematic parameters and the thermodynamic ones {is set} to $1$. The integration in time is done using the forward Euler method but can be extended straightforwardly to higher-order methods. In order to guarantee that the limit solution will be a weak solution of the problem, we introduce a general correction method in the spirit of the Lagrangian staggered method described in \cite{Svetlana,MR4059382, MR3023731}, and we prove a Lax Wendroff theorem. The proof is valid for multidimensional versions of the scheme, even though most of the numerical illustrations in this work, on classical benchmark problems, are one-dimensional because we have easy access to the exact solution for comparison. We conclude by explaining that the method is general and can be used in different settings, for example, Finite Volume, or discontinuous Galerkin method, not just the specific one presented in this paper.
\end{abstract}
\section{Introduction}
The Euler equations of fluid dynamics are, formulated in their conservative version, 
\begin{equation}\begin{split}
\dfrac{{\partial \rho}}{{\partial t}}&+\mathrm{div}(\rho \mathbf{u}) =\mathbf{0},  \\
\dfrac{{\partial \rho \mathbf{u}}}{{\partial t}}&+\mathrm{div}(\rho \mathbf{u}\otimes \mathbf{u} + p\mathbf{I}) =\mathbf{0},  \\
\dfrac{{\partial E}}{{\partial t}}&+\mathrm{div}((E+p) \mathbf{u}) =\mathbf{0}.
\end{split}
\label{EulerCons}
\end{equation}%
As usual, $\rho\geq 0$ is the density, $\bu$ is the velocity vector, $E=e+\tfrac{1}{2}\rho \bu^2$ is the total energy, $e\geq 0$ is the internal energy and $p$ is the pressure. The system is closed by an equation of state for $p=p(\rho, e)$. The simplest one is that of a calorically perfect gas $$p=\frac{e}{\gamma-1},$$ where the ratio of specific heats $\gamma$ is constant.  

When the solution is smooth, the system (\ref{EulerCons}) can be equivalently written in nonconservative form as
\begin{equation}\begin{split}
\dfrac{{\partial \rho}}{{\partial t}}&+\mathrm{div}(\rho \mathbf{u}) =\mathbf{0},  \\
\dfrac{{\partial  \mathbf{u}}}{{\partial t}}&+(\mathbf{u} \cdot \nabla)\mathbf{u} + \dfrac{\nabla p}{\rho} =\mathbf{0},  \\
\dfrac{{\partial e}}{{\partial t}}&+\mathbf{u} \cdot \nabla\mathbf{e} + (e+p)\mathrm{div}\mathbf{u} =\mathbf{0}.  
\end{split}
\label{EulerNonCons}
\end{equation}%
When the solution is not smooth, the form \eqref{EulerNonCons} is meaningless because the differential operators are no longer defined. This is why the form \eqref{EulerCons} is preferred, in particular in its weak form, see \cite{raviart1}. This fact has a very strong implication for the design of numerical schemes applied to \eqref{EulerCons}: the Lax-Wendroff theorem implies and guarantees that a suitable numerical approximation should be written in terms of flux. 

However, the form \eqref{EulerNonCons} is better suited for engineering purposes, since one has direct access to the velocity and the internal energy. Hence a rather natural question is how to discretise the Euler equations directly from \eqref{EulerNonCons}, and still have convergence to the correct weak solutions, at least formally.
{\color{black}In addition to this theoretical question, there are other reasons to use the \eqref{EulerNonCons} system, and we list a few of them:
\begin{itemize}
\item In the Lagrangian hydrodynamics community (i.e. the US National Laboratories, AWE in the UK, CEA in France and their Russian and Chinese counterparts), it is very common to describe, for certain applications, the equations of compressible fluid dynamics using volume, velocity and specific internal energy. This is, for example, what is done with the Wilkins scheme, see \cite{wilkins}, and even in the finite element version of this scheme, see \cite{MR3023731}.  Variables are represented either by cell or by point values, depending on whether they are intensive (such as velocity) or extensive (such as mass, volume and energy): this implies that thermodynamic variables are in $L^2$ only, while velocity is globally continuous (except when we need to introduce slip surfaces). 

We think it is interesting to understand how the mechanism that makes these schemes conservative can be translated into the Eulerian framework. The technique we develop in this article can be seen as the Eulerian counterpart of what was done in \cite{Svetlana} and then \cite{MR4059382} in the Lagrangian framework.

\item In many multi-physics applications, it is not obvious whether the fully conservative formulation is the most appropriate. MHD equations are a good example. In the finite-volume community, it is customary to describe flow with density, momentum and total energy, i.e. the sum of kinetic energy, thermodynamic energy and magnetic energy. The magnetic field evolution equation is described in conservation form (using Ohm's law and Faraday's equation), and we have the constraint $\text{div }B=0$.   However, the natural way to write the magnetic field equation is not using the divergence operator but the curl operator, and the consequence is the preservation of the divergence involution. Merging the magnetic field and the mechanical and thermodynamic energies may be considered somewhat artificial, since we also have an evolution relation for the magnetic field equation. Hence,  it may be considered interesting to separate the thermodynamics from the magnetic field. See \cite{tcheque} for an example of this type of approach in the Lagrangian framework. In addition, one way of preserving the structure of the Faraday equation is to use a staggered mesh. This is not necessarily done as in the present paper, but respecting local conservation of total energy may require the same kind of algebraic manipulations as here, see \cite{AbgrallDumbser}.
\item When considering a mixture of gases, the most natural variable describing fluid energy is not total energy but internal energy, or even pressure. Consequently, the use of a formulation based on a non-conservative form of the system may offer certain advantages.
\item One numerical strategy for simulating incompressible flows is to use a staggered mesh, for reasons of stability. It is well known in the finite volume community that when data are collocated and the Mach number tends towards 0, the behavior of the numerical scheme degrades considerably. This problem has been studied in numerous articles, \cite{herve3,herve2,herve1} for example, and the references therein. Several strategies exist, such as preconditioning, but not only. One might expect staggering, even for the  compressible case, to be a good strategy when the Mach number tends towards $0$. This was done by Herbin et al. in \cite{MR4232217} and by Bijl et al. in \cite{zbMATH01282684}. The scheme must be implicit. Here we are not interested in the low  Mach effect, but in conservation issues, and the extension of our work  to small  Mach numbers could be a possibility.
\end{itemize}
}
\medskip
{\color{black}We will  be using the expression "locally conservative" for a scheme. By this we mean that for finite volume schemes, discontinuous Galerkin schemes, finite differences and even those using continuous finite elements (see \cite{MR4090481} for a \emph{explicit} construction), each degree of freedom can be associated with a (control) volume.  We say that a scheme is locally conservative if, for any sub-domain obtained by gathering such volumes, the update of the conservative variable is obtained by the contribution on the boundary of this sub-domain.
}
\medskip

One obvious way to write a scheme on the primitive variables is to start  from a locally conservative approximation of \eqref{EulerCons}, and by simple algebraic manipulations which amount to multiplying the numerical scheme by approximations of 
\begin{equation}
\begin{pmatrix}1&0&0 \\
\mathbf{u}&\rho&0\\
\dfrac{\mathbf{u}^2}{2}& \rho \mathbf{u} & 1
\end{pmatrix},
\label{MatrixConsToNoncons}
\end{equation}%
we can obtain a scheme directly working on the primitive variables. This ``new'' scheme is \emph{equivalent} to the original one.

This is not exactly the question we want to address here. We are interested in designing locally conservative approximations of \eqref{EulerNonCons} for which the thermodynamic variables are approximated in $L^2$ while the velocity is globally continuous. This can be seen as an Eulerian version of the Lagrangian schemes designed in \cite{Svetlana,MR4059382} or \cite{MR3023731} and the related works by these authors. A similar question has been addressed by Herbin and co-authors, see, for example, \cite{MR3864518,MR4122491,MR4232220} in the Finite Volume context. In these references, the authors describe a class of numerical schemes where the thermodynamic variables and the velocity are piecewise constant but logically described on a staggered mesh. They show the convergence towards the weak solution. The scheme can also be partially implicit, so that in the low Mach number limit the scheme ``degenerates'' to a Mac-type scheme, see \cite{MR4232217}. \color{black}{Their schemes are second-order accurate in time and space.}

{\color{black}
In this article, we describe a different technique that allows \textit{a priori} to achieve an arbitrary level of precision, both in time and space. This technique is not particularly designed for any specific class of scheme. The main restriction seems to be that the time scheme must be based on a sequence of Euler steps. Examples are the Runge Kutta SSP schemes, or the Defect Correction (DeC) methods in the \cite{zbMATH06826581} version. We have chosen to illustrate the technique on the example of residual distribution (RD) schemes where evolution over time is done by DeC, see \cite{zbMATH06826581}. This choice is also motivated by the fact that local conservation recovery is simpler for RD schemes (see \cite{MR4090481}). Therefore, before describing this method, we briefly review the class of residual distribution schemes that will be the main tool we use, and sketch how dG (and therefore finite-volume) schemes can be reformulated in this framework. We then describe the scheme and explain why it is locally conservative. Finally, we show a variant of the Lax-Wendroff theorem adapted to our framework. Finally, we show how to adapt the method to finite volume schemes and discontinuous Galerkin methods. Numerical examples illustrate the soundness of the approach.
}

\section{A first-order nonconservative approach}
We have in mind a numerical approximations where the variables are piecewise polynomial in simplex. We also assume that the velocity is globally continuous, in contrast to discontinuous Galerkin (dG)--like approximations. This constraint is motivated by the \emph{choice} that we want to extend the technique of \cite{Svetlana}, where a Petrov Galerkin technique is used, inspired from \cite{zbMATH06826581} and the reference therein. If nothing special is done, we need to invert a mass matrix. This can be cumbersome, and even impossible if we want to extend the techniques of \cite{AbgrallRoe} because the equivalent of the mass matrix changes at every time step. This is why a particular time stepping should be preferred, for example, the Deferred Correction (DeC) approach, see Appendix \ref{appendix_dec}. It relies on series of Euler forward type of discretisation. 

This is indeed \textbf{\emph{the}} essential point: if one prefers to forget the globally continuous methods, rely on a dG--like approach, and use a Strong Stability Preserving (SSP) Runge-Kutta approach, one can extend our correction technique and build schemes that converge to a weak solution of the problem, starting from \eqref{EulerNonCons}.  {This will be explained in section \ref{sec:dG}.} Since the novelty of the approach lies in the correction technique, we will focus for simplicity on a single Euler forward step in time.

We consider a hyperbolic system in the form
\begin{equation}\label{Hyperbolic}
\dfrac{\partial U}{\partial t} + \mathrm{L}(U) = 0
\end{equation} 
on a domain $\Omega\subset \R^d$, $d=1,2,3$. For solving \eqref{EulerCons} or \eqref{EulerNonCons} we define
\begin{equation}\label{L:cons}L(U)=: \begin{pmatrix}\mathrm{div}(\rho \mathbf{u}) \\ \mathrm{div}(\rho \mathbf{u}\otimes \mathbf{u} + p\mathbf{I}) \\ \mathrm{div}((E+p) \mathbf{u})\end{pmatrix}\end{equation}
for the conservative form and
\begin{equation}\label{L:ncons}L(U)=:\begin{pmatrix}\mathrm{div}(\rho \mathbf{u})\\   (\mathbf{u} \cdot \nabla)\mathbf{u} + \dfrac{\nabla p}{\rho}\\ \mathbf{u} \cdot \nabla\mathbf{e} + (e+p)\mathrm{div}\mathbf{u}\end{pmatrix} \end{equation}
for the non--conservative form. In what follows, we will describe the procedure for solving the equations in two steps. First, we consider the case of a scalar problem, and then we look at \eqref{EulerCons} or \eqref{EulerNonCons}. The reason is that in \eqref{EulerNonCons} not all of the variables play the same role, contrarily to \eqref{EulerCons}, and it is easier to start with a system with one variable. For now, we will proceed by forgetting the question of local conservation.

\subsection{Scalar case}\label{2.2}
\subsubsection{Trial space}
We consider a triangulation of $\Omega$ made of non-overlapping simplices that are generically denoted by $K$. We assume that the triangulation is conformal, and define
$${V^h(\Omega)=\{ v\in L^2(\Omega) \text{ such that for any }K, v_{\vert K}\in \P^k(\R^d)\}\subset L^2(\Omega)},$$ where as usual, $\P^k(\R^d)$ is the set of polynomials in $\R^d$ of degree less or equal to $k$. We also define
$${W^h(\Omega)=V^h(\Omega)\cap C^0(\Omega)}.$$
 In each element $K$, a polynomial is defined by a set of degrees of freedom, for example, the Lagrange points. We denote by $\sigma$ a generic degree of freedom. Here, for reasons that will be more clear later on, we expand the polynomials in terms of B\'ezier polynomials. 
\begin{itemize}
\item One dimensional elements:
In the element $K=[x_i,x_{i+1}]$, we consider the barycentric coordinates
$$\lambda_1(x)=\dfrac{x_{i+1}-x}{x_{i+1}-x_i}, \quad \lambda_2(x)=\dfrac{x-x_i}{x_{i+1}-x_i}=1-\lambda_1.$$
If $\sigma\in K$, the restriction of $B_\sigma$ is defined in the element as follows: if $\sigma\not\in K$, then the B\'ezier form vanishes.
We describe the two families of B\'ezier forms we will need:
\begin{itemize}
\item Linear: The degrees of freedom are the vertices, so 
$$\varphi_i^{{n+1}}=\lambda_1, \quad \varphi_{i+1}^{{n+1}}=\lambda_2\; \text{ and }\; \sigma=x_i \text{ or } x_{i+1} \text{ here.}$$
\item Quadratic: The degrees of freedom $\sigma$ are identified with the vertices $i$, and the mid-points $i+\tfrac{1}{2}$ 
$$\varphi_\sigma^{(2)}(x)=\left \{\begin{array}{ll}\lambda_1^2, & \text{ if }\sigma=x_i,\\
2\lambda_1\lambda_2, & \text{ if } \sigma=x_{i+\nicefrac{1}{2}},\\
\lambda_2^2, & \text{ if } \sigma=x_{i+1}.
\end{array}
\right .$$
\end{itemize}

\item Multidimensional elements: We only describe the 2D cases, with triangles, but similar things are obtained for quadrangles, or 3D simplices.
A triangle is made of three vertices denoted by $1$, $2$ and $3$. The barycentric coordinates with respect to the vertices $1$, $2$, $3$ are denoted by $\Lambda_1$, $\Lambda_2$ and $\Lambda_3$.
\begin{itemize}
\item Linear: The degrees of freedom are the vertices and $\varphi_{\sigma_i}=\Lambda_i$ and $i=1,2,3$.
\item Quadratic: The degrees of freedom are the three vertices $\sigma_1$, $\sigma_2$ and $\sigma_3$ as well as the midpoints of the edges: 
$$\sigma_4=\frac{\sigma_1+\sigma_2}{2}, \quad \sigma_5=\frac{\sigma_2+\sigma_3}{2}, \quad\sigma_6=\frac{\sigma_3+\sigma_1}{2}.$$
The B\'ezier polynomials are:
\begin{equation*}
\begin{split}
&\varphi_{\sigma_i}=\Lambda_i^2\;  \text{ for }\;i=1, 2,3, \\
&\varphi_{\sigma_4}=2\Lambda_1\Lambda_2, \quad 
\varphi_{\sigma_5}=2\Lambda_3\Lambda_2,  \quad \varphi_{\sigma_6}=2\Lambda_1\Lambda_3.
\end{split}
\end{equation*}
\end{itemize}
\end{itemize}
Then, considering $u\in V_h(\Omega)$ or $u\in W_h(\Omega)$, for any $K$, we expand $u_{\vert K}$ as
$$u_{\vert K}=\sum_{\sigma\in K} u_\sigma \varphi_\sigma^K,$$ where $\varphi_\sigma^K$ is any of the linear, quadratic (or higher-order) functions defined above. If $u\in V_h(\Omega)$ then we have the expansion
$$u=\sum_K\sum_{\sigma\in K}u_\sigma^K \varphi_\sigma^K$$
and if $u\in W_h(\Omega)$, we can expand $u$ as
$$u=\sum_\sigma u_\sigma\varphi_\sigma.$$ With some abuse of notations, we will use the second expansion throughout this paper, depending if we see $\varphi_\sigma$ per element or more globally.
\subsubsection{Test space}
As we mentioned earlier, we rely on a Petrov-Galerkin approach. This means that the test functions will belong to a finite dimensional subspace $X_h(\Omega)$ of $L^2(\Omega)$ that can also be described by the degrees of freedom $\sigma$: we can identify functions of $X_h(\Omega)$ that are indexed by the $\sigma$ and span this space. We denote them by $\Xi_\sigma$. For example, in the SUPG method, we define $\Xi_\sigma$ in each $K$ by: for $\bx\in K$, 
$$\Xi_\sigma(\bx)=\varphi_\sigma(\bx)+h_K\big (\nabla_U L(U)\tau_K\big )\cdot \nabla\varphi_\sigma(\bx).$$ Here $h_K$ is the diameter of $K$ and $\tau_K$ is a positive matrix. In \cite{AbgrallRoe,energy,larat,santis1,santis2}, examples are given, where $\Xi_\sigma$ depends on the solution, in order to get $L^\infty$ stability. In all the examples we are considering, the support of $\Xi_\sigma$ is that of $\varphi_\sigma$.
\subsubsection{Description of the time discretisation}
We start by integrating (\ref{Hyperbolic}) which gives
\begin{equation*}
\int_\Omega\int_{t^{n}}^{t^{n + 1}}\Xi_\sigma\left(\dfrac{\partial U}{\partial t} + \mathrm{L}(U)\right) \; dt \; d\bx= 0.
\end{equation*}
By applying the forward Euler method per simplex, we obtain
\begin{equation}
\label{forw_Euler}
\int_\Omega\int_{t^{n}}^{t^{n + 1}}\Xi_\sigma\left(\dfrac{\partial U}{\partial t} + \mathrm{L}(U)\right) \; dt \; d\bx = \int_\Omega \Xi_\sigma (U^{n+1}-U^n)\; d\bx + \Delta t \int_\Omega \Xi_\sigma \mathrm{L}(U^n) \; d\bx =0.
\end{equation}
{\color{black}The definition of $\int_K \Xi_\sigma\mathrm{L}(u) d\bx$ is somewhat formal and we replace it by some approximation $\Phi_\sigma^K(U)$ which will be defined later. {\color{black}The only constraint is that we have the relation
\begin{equation}\label{pseudo:conservation}
\sum\limits_{\sigma\in K} \Phi_\sigma^K(U)=\Phi^K(U),
\end{equation}
where the precise definition of $\Phi^K(U)$ depends on whether we are dealing with the problem \eqref{Hyperbolic} with the $L$ operator in conservation form $L(U)=\text{div }\bbf(U)$ as in \eqref{L:cons} or 
or in non conservation form $L(U)=\mathbf{a}(U)\cdot \nabla U$ as in \eqref{L:ncons}.}
More specificaly, 
\begin{itemize}
\item If $L$ is in conservation form,  we  set
$$\Phi^K(U):=\int_{\partial K} \hbbf_\bn\; d\gamma,$$ where $\hbbf_\bn$ is a consistant approximation of the flux $\bbf$ in the direction $\bn$ (normal to $\partial K$),
\item If $L$ is in non conservation form, we set 
$$\Phi^K(U):=\int_K \mathbf{a}(U^h)\cdot \nabla U^h\; d\bx$$  where  a quadrature formula is employed.
In fact, and in that case, the situation is slightly more complicated, because written as such, there might seem there is no coupling between elements. Since this is a  case by case procedure, we give an example in section \ref{numerics}.
%
\end{itemize}
}
We replace the temporal terms by
$$\int_\Omega \varphi_\sigma (U^{n+1}-U^n) \; d\bx$$ and then "lump" the mass matrix, set $\int_K\varphi_\sigma dx=C_K$ (it does not depend on $\sigma$) and obtain
\begin{equation}
\label{semi_forw_Euler}
|C_\sigma| ( U^{n+1}-U^n ) + \Delta t \sum_{K,\sigma \in K} \Phi_\sigma^K(U^n) = 0.
\end{equation}
We note that this is the reason why we use a B\'ezier approximation since we are sure that the lumped mass is non zero because it holds
$$C_\sigma=\int_\Omega \varphi_\sigma\; d\bx>0.$$

{\color{black} This scheme corresponds to the $\mathcal{L}^1$ operator of the DeC procedure described in appendix \ref{appendix_dec}. It will be high order in time and space, provided some conditions described in appendix \ref{appendix_dec} are fullfiled.  We stick to this, to avoid  useless complications, and also because its form is that of an Euler forward method. Hence our discussion becomes valid for any algorithm that can be put in the  form \eqref{semi_forw_Euler} with  $\Phi_\sigma^K$ satisfying \eqref{pseudo:conservation}.

Let us give an other example: the discontinuous Galerkin method in the conservative setting. Using basis functions $\{\varphi^\sigma\}$ that are now see as polynomial in each $K$ but only in $L^2$, we have for any $\sigma\in K$
$$\int_K\varphi_\sigma \dpar{U}{t}\; d\bx -\int_K \nabla \varphi_\sigma \cdot\bbf(U)\; d\bx+\int_{\partial K} \varphi_\sigma\hbbf_\bn \;d\gamma=0$$
and we set
$$\Phi_\sigma^K :=-\int_K \nabla \varphi_\sigma \cdot\bbf(U)+\int_{\partial K} \varphi_\sigma\hbbf_\bn \; d\gamma.$$
Since in $K$, $\sum\limits_{\sigma\in K}\varphi_\sigma=1$, we have \eqref{pseudo:conservation}. The non conservative setting works \emph{formally} similarly, provided a case by case strategy is again adopted.
}

\subsection{Case of system \eqref{EulerNonCons}}
We describe the residuals and develop the method as before for simplicity for the forward Euler method in time. But as mentioned before, it can be extended to higher orders in a straightforward way.

\medskip

We assume that the computational domain $\Omega$ is covered by non-overlapping simplices $\{K_j\}_{j\in \mathcal{T} }.$ The velocity field $\mathbf{u}$ belongs to a kinematic space $\mathcal{V}$ of finite dimension; it has a basis denoted by $\{\varphi_{\sigma_{\mathcal{V}}}\}_{\sigma_{\mathcal{V}}\in D_{\mathcal{V}}}$, where $D_{\mathcal{V}}$ is the set of kinematic degrees of freedom with the total degrees of freedom given by $\# D_{\mathcal{V}} = N_{\mathcal{V}}$. The thermodynamic quantities such as the internal energy, the density and the pressure belong to a thermodynamic space $\mathcal{E}$; this space is also finite dimensional and its basis is $\{\varphi_{\sigma_{\mathcal{E}}}\}_{\sigma_{\mathcal{E} \in D_{\mathcal{E}}}}$. The set $D_{\mathcal{E}}$ is the set of thermodynamic degrees of freedom with the total degrees of freedom $\# D_{\mathcal{E}} = N_{\mathcal{E}}$. The kinematic space $\mathcal{V}$ is formed by the quadratic (or linear) Bernstein elements, while the thermodynamic space $\mathcal{E}$ has a piecewise-linear (or piece-wise constat) basis. The velocity field is approximated by 
$$\mathbf{u}(\mathbf{x}, t) = \sum_{\sigma_{\mathcal{V}}\in D_{\mathcal{V}}} \mathbf{u}_{\sigma_{\mathcal{V}}}(t)\varphi_{\sigma_{\mathcal{V}}}(\mathbf{x}),$$
where the $\varphi_{\sigma_{\mathcal{V}}}$ are the linear/quadratic (or linear) B\'ezier polynomials, 
and the density, the pressure and the internal energy, are given by
 \begin{equation*}
 \begin{split}
  \rho(\mathbf{x}, t) &= \sum_{\sigma_{\mathcal{E}}\in D_{\mathcal{E}}} \rho_{\sigma_{\mathcal{E}}}(t)\varphi_{\sigma_{\mathcal{E}}}(\mathbf{x}), \quad  p(\mathbf{x}, t) = \sum_{\sigma_{\mathcal{E}}\in D_{\mathcal{E}}} p_{\sigma_{\mathcal{E}}}(t)\varphi_{\sigma_{\mathcal{E}}}(\mathbf{x}),\\ e(\mathbf{x}, t) &= \sum_{\sigma_{\mathcal{E}}\in D_{\mathcal{E}}} e_{\sigma_{\mathcal{E}}}(t)\varphi_{\sigma_{\mathcal{E}}}(\mathbf{x}),
 \end{split}
 \end{equation*}
where the $\varphi_{\sigma_{\mathcal{E}}}$ are the per elements piecewise constant/linear functions. Note that the degrees of freedom for the velocity are assumed to be globally continuous, so in {$\big (W_h(\Omega)\big )^{d}$}, while the thermodynamic ones are discontinuous across the boundary of the elements, so in {$\big (V_h(\Omega)\big )^2$.}
 
We can rewrite the Euler equations (\ref{EulerNonCons}) in the following way $$\dfrac{\partial U}{\partial t} + \mathbf{a}(U)\cdot \nabla U = 0.$$
 
The only thing to do is to describe how the method of the previous section adapts to this case, and this amounts to describing the general structure residuals, that is how \eqref{pseudo:conservation} is written. Since the velocity is globally continuous, 
 we write 
 $${\Phi^{K, \bu}=\int_K \big ( \bu\otimes \bu+\dfrac{\nabla p}{\rho}\big )\; d\bx},$$
where {$\bu\in \big (W_h(\Omega)\big )^d$ }and {$p,\rho\in V_h(\Omega)$}. Since we are on $K$, these are simply polynomials, and the integration is carried out by numerical quadrature.

For the density, the evolution equation  is in conservation form and we use a numerical flux $\hbbf$:
 $${\Phi^{K, \rho}=\int_{\partial K}\hbbf_\bn\; d\gamma.}$$
Here, any consistent numerical flow can be used a priori. Of course, the stability of the method depends on this choice, but the conservation properties of the method do not.
 
Last, for the internal energy, we write
 $${\Phi^{K, E}=\int_K \big ( \bu\cdot \nabla e+ (e+p)\text{ div } \bu \big ) \; d\bx}$$ and again we use a quadrature formula. 
 
 In the numerical section, we will describe the residuals that we use.

\section{A discussion on conservation}\label{discussion}
\subsection{A set of sufficient conditions to achieve convergence to a weak solution}
Again, to simplify the notations, we focus on the first-order case, but the extension to the more general case is straightforward. In the appendix \ref{LxWthm}, we show a Lax Wendroff theorem for this type of discretisation. What we do here is to show how to go from the system in non--conservative form to the one in conservative form. 

Nothing has to be done for the density since it is already in conservative form and the standard proof \cite{raviart1} applies. There is no need to repeat it here since the proof we give for the momentum and the total energy, modulo some complications, is essentially the same.

Let us first look at the momentum. Considering a test function $\psi\in C^1_0(\R^d\times \R)$, we denote with $\psi^n_K$ the value of $\psi$ at time $t_n$ at the centroid of $K$, and consider the following approximation of $\psi$ that we still denote by $\psi$:
$$\psi(\bx,t)=\sum_K \psi_K^n 1_K\quad \text{ for }\quad t\in [t_n, t_{n+1}[.$$
Then we consider
\begin{equation}
\label{momentum:1}
\begin{split}
\int_{\R^d}\psi(\bx,t) \; &\big ( \rho^{{n+1}}\bu^{{n+1}}-\rho^{{n}}\bu^{{n}}\big ) \; d\bx=\sum_K\psi_K^n \int_K\big ( \rho^{{n+1}}\bu^{{n+1}}-\rho^{{n}}\bu^{{n}}\big ) \; d\bx\\&=\sum_K \psi_K^n\bigg [ \int_K \rho^{{n+1}} \big ( \bu^{{n+1}}-\bu^{{n}}\big ) \; d\bx+\int_K \bu^{{n}} \big ( \rho^{{n+1}}-\rho^{{n}}\big ) \; d\bx\bigg ].
\end{split}
\end{equation}
Introducing $\Delta \bu_{\sigma_\mathcal{V}}=:\bu_{\sigma_\mathcal{V}}^{{n+1}}-\bu_{\sigma_\mathcal{V}}^{{n}}$ and
$\Delta \rho_{\sigma_{\mathcal{E}}}:=\rho_{\sigma_{\mathcal{E}}}^{{n+1}}-\rho_{\sigma_{\mathcal{E}}}^{{n}}$, we can write
$$\int_K \rho^{{n+1}} \big ( \bu^{{n+1}}-\bu^{{n}}\big )\; d\bx=\sum_{\sigma_{\mathcal{V}}\in K} \Delta \bu_{\sigma_\mathcal{V}}\int_K \rho^{{n+1}}\varphi_{\sigma_\mathcal{V}}\; d\bx$$
and 
$$\int_K \bu^{{n}} \big ( \rho^{{n+1}}-\rho^{{n}}\big )\; d\bx= \sum_{\sigma_{\mathcal{E}}\in K}\Delta \rho_{\sigma_{\mathcal{E}}}\int_K \bu^{{n}}\varphi_{\sigma_\mathcal{E}}\; d\bx.$$

Hence, \eqref{momentum:1} can be rewritten as:
\begin{equation}
\label{momentum:2}
\begin{split}
\int_{\R^d}&\psi(\bx,t) \; \big ( \rho^{{n+1}}\bu^{{n+1}}-\rho^{{n}}\bu^{{n}}\big ) \; d\bx\\
& =\sum_K\psi_K^n\bigg [\sum_{\sigma_{\mathcal{V}}\in K} \Delta \bu_{\sigma_\mathcal{V}}\int_K \rho^{{n+1}}\varphi_{\sigma_\mathcal{V}}\; d\bx+ \sum_{\sigma_{\mathcal{E}}\in K}\Delta \rho_{\sigma_{\mathcal{E}}}\int_K \bu^{{n}}\varphi_{\sigma_\mathcal{E}}\; d\bx
\bigg ]\\
&\quad=\sum_K\psi_K^n\bigg [\sum_{\sigma_{\mathcal{V}}\in K} \omega^{\rho,n+1,K}_{{\sigma_{\mathcal{V}}}} 
|C_{\sigma_{\mathcal{V}}}|\Delta \bu_{\sigma_\mathcal{V}}
 + \sum_{\sigma_{\mathcal{E}}\in K}\omega^{\bu,n,K}_{\sigma_\mathcal{E}}
|C_{\sigma_\mathcal{E}}|\Delta \rho_{\sigma_{\mathcal{E}}}\bigg ]\\
&=\sum_K\Big [\sum_{\sigma_\V\in K} \psi_{\sigma_\V}\omega_{\sigma_\V}^{\rho, n+1,K}|C_{\sigma_{\mathcal{V}}}|
\Delta \bu_{\sigma_\mathcal{V}} \Big ]\\
&\qquad +
\sum_K\sum_{\sigma_\V} \big (\psi_K^n-\psi_{\sigma_\V}\big )\vert C_{\sigma_{\mathcal{V}}}\vert \; \omega_{\sigma_\V}^{\rho, n+1}
\Delta \bu_{\sigma_\mathcal{V}}+\sum_K \psi^n_K \Big [\sum_{\sigma_{\mathcal{E}}\in K}\omega^{\bu,n,K}_{\sigma_\mathcal{E}}
|C_{\sigma_\mathcal{E}}|\Delta \rho_{\sigma_{\mathcal{E}}}\Big ]\\
&=
\sum_{\sigma_\V} \psi_{\sigma_\V}^n \omega_{\sigma_\V}^{\rho, n+1}\vert C_{\sigma_{\mathcal{V}}}\vert \;\Delta \bu_{\sigma_\V} +
\sum_K\sum_{\sigma_\V\in K} \big (\psi_K^n-\psi_{\sigma_\V}\big )\vert C_{\sigma_{\mathcal{V}}}\vert \; \omega_{\sigma_\V}^{\rho, n+1}
\Delta \bu_{\sigma_\mathcal{V}}
\\
&\qquad+\sum_K \psi^n_K \Big [\sum_{\sigma_{\mathcal{E}}\in K}\omega^{\bu,n,K}_{\sigma_\mathcal{E}}|C_{\sigma_\mathcal{E}}|\Delta \rho_{\sigma_{\mathcal{E}}}\Big ]
\end{split}
\end{equation}
where we have set for simplicity
$$\omega^{\rho,n+1,K}_{{\sigma_{\mathcal{V}}}}:=\dfrac{\int_K \rho^{{n+1}}\varphi_{\sigma_\mathcal{V}}\; d\bx}{|C_{\sigma_{\mathcal{V}}}|}\; \text{ , }\;
\omega^{\bu,n,K}_{\sigma_\mathcal{E}}:=\dfrac{\int_K \bu^{{n}}\varphi_{\sigma_\mathcal{E}}\; d\bx}{|C_{\sigma_\mathcal{E}}|}$$
and (using that the support of $\varphi_{\sigma_\V}$ is the union of the elements that share $\sigma_\V$),
$$\omega^{\rho,n+1}_{{\sigma_{\mathcal{V}}}}:=\dfrac{\int_\Omega \rho^{{n+1}}\varphi_{\sigma_\mathcal{V}}\; d\bx}{|C_{\sigma_{\mathcal{V}}}|}$$


For the velocity, we have:
$$|C_{\sigma_{\mathcal{V}}}|\big ( \bu_{\sigma_\mathcal{V}}^{{n+1}}-\bu^{{n}}_{\sigma_\mathcal{V}}\big ) +\Delta t_n\sum_{K, \sigma_\mathcal{V}\in K}\Phi^\bu_{\sigma_\mathcal{V},K}=0,$$
where, for the forward Euler scheme \eqref{forw_Euler}, 
$$\Phi^\bu_{\sigma_\mathcal{V},K}=\Phi^\bu_{\sigma_\mathcal{V},K}(U^{{n}}
\big ).$$

For the density, we have
$$\vert C_{\sigma_\mathcal{E}}\vert \big ( \rho_{\sigma_{\mathcal{E}}}^{{n+1}}-\rho_{\sigma_{\mathcal{E}}}^{{n}}\big ) +\Delta t_n
\sum_{K, \sigma_{\mathcal{E}}\in K} \Phi_{\sigma_{\mathcal{E}},K}^{\rho}=0$$
and we note that the sum reduces to one term, hence
$$\vert C_{\sigma_\mathcal{E}}\vert \big ( \rho_{\sigma_{\mathcal{E}}}^{{n+1}}-\rho_{\sigma_{\mathcal{E}}}^{{n}}\big ) +\Delta t_n 
\Phi^\rho_{\sigma_{\mathcal{E}},K}=0,$$
where again
$$\Phi^\rho_{\sigma_{\mathcal{E}},K}=\Phi^\rho_{\sigma_\mathcal{E},K}(U^{{n}})$$and $K$ is \emph{the} element such that $\sigma_\mathcal{E}\in K.$
Using these relations in \eqref{momentum:2}, we get
\begin{equation}
\label{momentum:3}
\begin{split}
\int_{\R^d}&\psi(\bx,t) \; \big ( \rho^{{n+1}}\bu^{{n+1}}-\rho^{{n}}\bu^{{n}}\big ) \; d\bx+\Delta t_n\underbrace{
\sum_{\sigma_\V} \psi_{\sigma_\V}^n \omega_{\sigma_\V}^{\rho, n+1}
\bigg [ \sum_{K, \sigma_\V\in K}\Phi^\bu_{\sigma_\bV,K}\bigg ]}_{I}\\&
+
\Delta t_n\sum_K\bigg \{\sum_{\sigma_\V\in K} \big (\psi_K^n-\psi_{\sigma_\V}\big )\omega_{\sigma_\V}^{\rho, n+1,K}\Big [ \sum_{K', \sigma_\bV\in K'\cap K} \Phi^\bu_{\sigma_b,K}\Big ]\bigg \}
\\
&\qquad+
\Delta t_n\sum_K \psi^n_K \Big [\sum_{\sigma_{\mathcal{E}}\in K}\omega^{\bu,n,K}_{\sigma_\mathcal{E}}\Phi^\rho_{\sigma_\E,K}\Big ]
\\
&\qquad\qquad=0
\end{split}
\end{equation}
The term $I$ can be rewritten as 
\begin{equation*}
\begin{split}
\sum_{\sigma_\V} &
\psi_{\sigma_\V}^n \omega_{\sigma_\V}^{\rho, n+1}
\bigg [ \sum_{K, \sigma_\V\in K}\Phi^\bu_{\sigma_\bV,K}\bigg ]=
\sum_{K} \psi^n_K \sum_{\sigma_\V\in K} \omega_{\sigma_\V}^{\rho, n+1}\Phi^\bu_{\sigma_\bV,K}\\
&\qquad 
+
\sum_K \bigg [ \sum_{\sigma_\V\in K} \big ( \psi^n_K-\psi^n_{\sigma_\V}\big )\omega_{\sigma_\V}^{\rho, n+1}\Phi^\bu_{\sigma_\bV,K}
\bigg ]
\end{split}
\end{equation*}
Hence, gathering all together, we get
\begin{equation*}
\begin{split}
\int_{\R^d}&\psi(\bx,t) \; \big ( \rho^{{n+1}}\bu^{{n+1}}-\rho^{{n}}\bu^{{n}}\big ) \; d\bx\\
 &+\Delta t_n
\sum_{K} \psi^n_K\Bigg [  \sum_{\sigma_\V\in K} \omega_{\sigma_\V}^{\rho, n+1}\Phi^\bu_{\sigma_\bV,K}
\sum_{\sigma_{\mathcal{E}}\in K}\omega^{\bu,n,K}_{\sigma_\mathcal{E}}\Phi^\rho_{\sigma_\E,K}\Bigg ]\\\qquad &
+\Delta t_n
\sum_K \bigg [ \sum_{\sigma_\V\in K} \big ( \psi^n_K-\psi^n_{\sigma_\V}\big )\omega_{\sigma_\V}^{\rho, n+1}\Phi^\bu_{\sigma_\bV,K}
\bigg ]\\
&\qquad \qquad
+\Delta t_n
\sum_K\bigg \{\sum_{\sigma_\V\in K} \big (\psi_K^n-\psi_{\sigma_\V}\big )\omega_{\sigma_\V}^{\rho, n+1,K}\Big [ \sum_{K', \sigma_\bV\in K'\cap K} \Phi^\bu_{\sigma_b,K}\Big ]\bigg \}=0
\end{split}
\end{equation*}

Thus, we obtain the master equation:
\begin{subequations}
\begin{equation}
\label{master:1}
\begin{split}
 \int_{\R^d}\psi(\bx,t) \; &\big ( \rho^{{n+1}}\bu^{{n+1}}-\rho^{{n}}\bu^{{n}}\big ) \; d\bx 
 \\&+\Delta t\sum_K \psi_K^n \Bigg [ 
\sum_{\sigma_\mathcal{V}\in K} \omega^{\rho,n+1}_{\sigma_\mathcal{V}} \Phi_{\sigma_\mathcal{V},K}^\bu 
+\sum_{\sigma_{\mathcal{E}}\in K}\omega^{\bu,n,K}_{\sigma_\mathcal{V}}
\Phi^\rho_{\sigma_{\mathcal{E}},K}
 \Big ]\\
&\qquad+ \Delta t_n
\sum_{K}\Big ( {F}_K^\bm+ D_K^\bm
\Big )=0
\end{split}
\end{equation}
with {\color{black}
\begin{equation}
\label{master:2}
\begin{split}
F_K^\bm&=\sum_{\sigma_\V\in K} \big ( \psi^n_K-\psi^n_{\sigma_\V}\big )\omega_{\sigma_\V}^{\rho, n+1}\Phi^\bu_{\sigma_\V,K}\\
D_{K}^\bm&=\sum_{\sigma_\V\in K} \big (\psi_K^n-\psi_{\sigma_\V}\big )\omega_{\sigma_\V}^{\rho, n+1,K}\Big [ \sum_{K', \sigma_\bV\in K'\cap K} \Phi^\bu_{\sigma_b,K}\Big ]\\
 \omega_{\sigma_\V}^{\rho,n+1,K}&=\dfrac{ \int_K \rho^{{n+1}}\varphi_{\sigma_\mathcal{V}}\; d\bx}{\vert C_{\sigma_\mathcal{V}}\vert}, \qquad
\omega_{\sigma_\mathcal{V}}^{\rho,n+1}=
\sum\limits_{K, \sigma_\V\in K}\omega_{\sigma_\mathcal{V}}^{\rho,n+1,K}
\\
\omega^{\bu,n,K}_{\sigma_\mathcal{E}}&=\dfrac{\int_K \bu^{{n}}\varphi_{\sigma_\mathcal{E}}\; d\bx}{\vert C_{\sigma_\mathcal{E}}\vert}. 
\end{split}
\end{equation}  }
\end{subequations}

\bigskip
Let us now consider the total energy. First, we remark that (with similar notations as before) the following holds:
$$\Delta( \rho \bu^2)=\bu^{{n+1}}\cdot \Delta( \rho \bu)+\rho^{{n}}\bu^{{n}}\cdot \Delta \bu.$$
Combined with
$$\Delta (\rho \bu)=\rho^{{n+1}}\Delta \bu+\bu^{{n}}\Delta \rho$$
we obtain
$$\Delta (\rho \bu^2)=\big ( \rho^{{n+1}}\bu^{{n+1}}+\rho^{{n}}\bu^{{n}}\big )\cdot \Delta \bu+\bu^{{n+1}}\cdot \bu^{{n}}\Delta \rho.$$
To simplify, we will set
$$\widetilde{\bm}=\frac{\rho^{{n+1}}\bu^{{n+1}}+\rho^{{n}}\bu^{{n}}}{2}, \quad \widetilde{q^2}=\bu^{{n+1}}\cdot \bu^{{n}}.$$

Using these relations, we see that
\begin{equation*}
\begin{split}
\sum_K\psi_K\int_K\Delta E \; d\bx&=\sum_K \psi_K \bigg ( \int_K \Delta e\; d\bx + \int_K \tilde{\bm}\cdot \Delta \bu\; d\bx +\frac{1}{2}\int_K \widetilde{q^2}\Delta \rho\; d\bx \bigg )\\
&= \sum_K \psi_K\bigg [ \int_K \Delta e\; d\bx +\sum_{\sigma_\mathcal{V}\in K} \Delta \bu_{\sigma_{\mathcal{V}}}\cdot \int_K \widetilde{\bm}\varphi_{\sigma_\mathcal{V}}\; d\bx
\\
&\qquad+ \frac{1}{2}\sum_{\sigma_{\mathcal{E}}\in K}\Delta \rho_{\sigma_\mathcal{E}}\int_K  \widetilde{q^2}\varphi_{\sigma_\mathcal{E}}\; d\bx\bigg ].
\end{split}
\end{equation*}
First, we notice that
$$\int_K \Delta e \; d\bx=-\Delta t\sum_{\sigma_\mathcal{E}\in K}\Phi_{\sigma_\mathcal{E},K}^e.$$
Introducing
$$\theta_{\sigma_\mathcal{V}}^{\bm, K}=\dfrac{\int_K\widetilde{\bm}\varphi_{\sigma_\mathcal{V}}\; dx}{\vert  C_{\sigma_\mathcal{V}}\vert}\quad \text{and}\quad \theta_{\sigma_\mathcal{E}}^{q^2, K}=\dfrac{\int_K\widetilde{q^2}\varphi_{\sigma_\mathcal{E}}\; dx}{\vert  C_{\sigma_\mathcal{E}}\vert},$$
we get
$$\sum_{\sigma_\mathcal{V}\in K} \Delta \bu_{\sigma_{\mathcal{V}}}\cdot \int_K \widetilde{\bm}\varphi_{\sigma_\mathcal{V}}\; d\bx=
-\Delta t \sum_{\sigma_\mathcal{V}\in K}\theta_{\sigma_\mathcal{V}}^{\bm, K}\cdot\Bigg ( \sum_{K', \sigma_\mathcal{V}\in K'} \Phi_{\sigma_\mathcal{V}, K'}^\bu\Bigg )$$
and because $\sigma_\mathcal{E}$ belongs to a single element we have
$$\sum_{\sigma_{\mathcal{E}\in K}}\Delta \rho_{\sigma_\mathcal{E}}\int_K  \widetilde{q^2}\varphi_{\sigma_\mathcal{E}}\; d\bx=
-\Delta t \sum_{\sigma_{\mathcal{E}}\in K}\theta_{\sigma_\mathcal{E}}^{q^2, K}\Phi_{\sigma_\mathcal{E},K}^\rho.$$
Then proceeding as for the velocity, and introducing
$$\theta_{\sigma_\mathcal{V}}^{\bm}=\dfrac{\sum\limits_{K, \sigma_\mathcal{V}\in K}\int_K\widetilde{\bm}\varphi_{\sigma_\mathcal{V}}\; dx }{\vert  C_{\sigma_\mathcal{V}}\vert}=\sum_{K, \sigma_\mathcal{V}\in K}\theta_{\sigma_\mathcal{V}}^{\bm, K},$$
we get

{\color{black}
\begin{equation}
\label{master:energy}
\begin{split}\int_{\R^d}& \psi(\bx,t) \big (E^{n+1}-E^n\big )d\bx
\\+&\Delta t_n\sum_K \psi_K\Bigg ( \sum_{\sigma_\mathcal{E}\in K}\Phi_{\sigma_\mathcal{E},K}^e+\sum_{\sigma_\mathcal{V}\in K} \theta_{\sigma_\mathcal{V}}^{\bm}\cdot \Phi_{\sigma_\mathcal{V},K}^\bu+\frac{1}{2}\sum_{\sigma_{\mathcal{E}}\in K}\theta_{\sigma_\mathcal{E}}^{q^2, K}\Phi^\rho_{\sigma_{\mathcal{E},K}}\Bigg )\\
& +\Delta t_n
\sum_K\underbrace{
\Big[ \sum_{\sigma_\V} \big (\psi^n_K-\psi_{\sigma_\V}^n\big ) \theta_{\sigma_\V}^{\bm,K}\cdot \bigg \{\sum_{K', \sigma_\V\in K'\cap K}\Phi_{\sigma_\V,K}^\bu\bigg \}\Big ]}_{D_K^E}\\
&\qquad +\Delta t_n\sum_K\underbrace{ \sum_{\sigma_\mathcal{V}\in K} 
\big (\psi_{\sigma_\mathcal{V}}-\psi_K^n\big ) \theta_{\sigma_\mathcal{V}}^{\bm, K}\cdot\Phi_{\sigma\mathcal{V},K}^\bu}_{:=F_K^E}=0.
\end{split}
\end{equation}
}
\bigskip

{\color{black} As it is customary, we say that a family of meshes is shape regular  if there exists $\alpha>0$ depending only on this family such that 
the ratio of the inner  and outer diameters of any element of any  mesh of this family is greater than   $\alpha$. }
We show the following result in Appendix \ref{LxWthm}:
\begin{proposition}\label{LxW}
Assume that the mesh $\mathcal{T}_h$ is shape regular, we denote by $h$ the maximum diameter of the element of the mesh.
For any $K$, the residuals $\Phi_{\sigma_\mathcal{E},K}^\rho$, $\Phi_{\sigma_\mathcal{E},K}^e$, $\Phi_{\sigma_\mathcal{V},K}^\bu$ are Lipschitz continuous functions of their arguments, with Lipschitz constant of the form $C\cdot h$, where $C$ only depends on $\alpha$ and the maximum norm of the solution.

Assume that we have a family of meshes $\mathcal{F}=\{\mathcal{T}_{h_n}\}$ with $\lim\limits_{n\rightarrow +\infty} h_n=0$. We denote by $(U_{h_n})_{n\geq 0}$ the sequence of functions fulfilling:
$$\text{ if }t\in [t_n, t_{n+1}[,\; U(\bx, t)=(\rho(\bx, t_n), \bu(\bx,t_n), e(\bx,t_n))^T$$
with, if $K$ is the element that exists almost everywhere such that $\bx\in K$,
$$\rho(\bx,t_n)=\sum_{\sigma_\mathcal{E}\in K} \rho_{\sigma_\mathcal{E}}^n\varphi_{\sigma_\mathcal{E}}(\bx), \quad e(\bx,t_n)= \sum_{\sigma_\mathcal{E}\in K} e_{\sigma_\mathcal{E}}^n\varphi_{\sigma_\mathcal{E}}(\bx),$$
and $$\bu(\bx,t_n)=\sum_{\sigma_\mathcal{V}}\bu_{\sigma_\mathcal{V}}^n\varphi_{\sigma_\mathcal{V}}(\bx).$$ Here $\{(\rho^n_{\sigma_\mathcal{E}}), (\bu^n_{\sigma_\mathcal{V}}),(e^n_{\sigma_\mathcal{E}})\}_{n\geq 0, \sigma_\mathcal{E}, \sigma_\mathcal{V}}$ are defined by the introduced scheme. 

We assume that the density, velocity and internal energy are uniformly bounded and that a subsequence converges in $L^2$ towards $(\rho, \bu, e)$, where $\rho, e\in L^2(\R^d\times [0,T])$ and $\bu\in \big (L^2(\R^d\times [0,T]))^d$ . 

 We also assume  that the residuals satisfy
 \begin{equation}\label{momentum}
\sum_{\sigma_\mathcal{V}\in K} \omega^{\rho,n+1}_{\sigma_\mathcal{V}} \Phi_{\sigma_\mathcal{V},K}^\bu 
+\sum_{\sigma_{\mathcal{E}}\in K}\omega^{\bu,n,K}_{\sigma_\mathcal{E}}
\Phi^\rho_{\sigma_{\mathcal{E}},K}=\int_{\partial K} \bbf^\bm(U^{{n}})\cdot \bn\; d\gamma
\end{equation}
and 
\begin{equation}\label{energy}
\sum_{\sigma_\mathcal{E}\in K}\Phi_{\sigma_\mathcal{E},K}^e+\sum_{\sigma_\mathcal{V}\in K} \theta_{\sigma_\mathcal{V}}^{\bm}\cdot \Phi_{\sigma_\mathcal{V},K}^\bu+\frac{1}{2}\sum_{\sigma_{\mathcal{E}}}\theta_{\sigma_\mathcal{E}}^{q^2, K}\Phi_{\sigma_{\mathcal{E}}}=\int_{\partial K}\bbf^E(U^{{n}})\cdot \bn\; d\gamma,
\end{equation}
where we have set
\begin{equation}\label{coeff:corr}
\begin{split}
\omega_{\sigma_\mathcal{V}}^{\rho,n+1}=\dfrac{\sum\limits_{K, \sigma_\mathcal{V}\in K} \int_K \rho^{{n+1}}\varphi_{\sigma_{\mathcal{V}}}\; d\bx}{\vert C_{\sigma_\mathcal{V}}\vert}, &\qquad \omega^{\bu,n,K}_{\sigma_\mathcal{V}}=\dfrac{\int_K \bu^{{n}}\varphi_{\sigma_\mathcal{E}}\; d\bx}{\vert C_{\sigma_\mathcal{E}}\vert},\\
\widetilde{\bm}=\frac{\rho^{{n+1}}\bu^{{n+1}}+\rho^{{n}}\bu^{{n}}}{2}, &\qquad \widetilde{q^2}=\bu^{{n+1}}\cdot \bu^{{n}},
\\
\theta_{\sigma_\mathcal{V}}^{\bm}=\dfrac{\sum\limits_{K, \sigma_\mathcal{V}\in K}\int_K\widetilde{\bm}\varphi_{\sigma_\mathcal{V}}\; dx }{\vert  C_{\sigma_\mathcal{V}}\vert},&\qquad
 \theta_{\sigma_\mathcal{E}}^{q^2, K}=\dfrac{\int_K\widetilde{q^2}\varphi_{\sigma_\mathcal{E}}\; dx}{\vert  C_{\sigma_\mathcal{E}}\vert}
\end{split}
\end{equation}
with the assumption that there exists $C$ independent of $n$, such that $\Delta t\leq C h$. { In \eqref{momentum} (resp. \eqref{energy}), $\bbf^\bm(U^{{n}})\cdot \bn$ (resp. $\bbf^E(U^{{n}})\cdot \bn$) is the momentum component of the normal flux (resp. its energy component).}

Then $V=(\rho, \rho \bu, e+\tfrac{1}{2}\rho\bu^2)$ is a weak solution of the problem.
\end{proposition}
\subsection{How to achieve   discrete conservation}\label{howto}
Since there is no ambiguity, we drop the dependency of the residuals with respect to the element.

 Given a set of residuals that satisfy \eqref{pseudo:conservation} also satisfy \eqref{momentum} and \eqref{energy}. In this section, we will show how to slightly modify the original scheme so that the new one will satisfy \eqref{pseudo:conservation},  \eqref{momentum} and \eqref{energy}, and hence if the scheme converges, we have convergence towards a weak solution.
 To achieve this, following \cite{Svetlana,paola,entropy}, we introduce the correction terms in the residuals. This needs to be done only for the velocity and the internal energy.
 
 Knowing at time $t_n$ the solution $(\rho^n, \bu^n, e^n)$ we obtain with the forward Euler step $(\rho^{n+1}, \bu^{n+1}, e^{n+1})$. For this, we first compute $\rho^{{n+1}}$ and then perform the update for the velocity and the energy:

 \medskip
\noindent \textbf{Momentum.}\par
We introduce a  correction $r^\bu_{\sigma,K}$ so that 
\begin{equation}
\label{correction:u}
\Psi_{\sigma_{V}}^\bu=\Phi^{\bu}_{\sigma_{V}}(U^{{n}})+r^\bu_{\sigma_\mathcal{V}}
\end{equation}
is such that
\eqref{momentum} holds true for the new set of residuals, 
i.e.
\begin{equation*}
\sum_{\sigma_\mathcal{V}\in K} \omega^{\rho,n+1}_{\sigma_\mathcal{V}} r_{\sigma_\mathcal{V}}^\bu=\int_{\partial K}\bbf^\bm(U^{{n}})\cdot \bn\; d\gamma-\bigg \{\sum_{\sigma_\mathcal{V}\in K} \omega^{\rho,n+1}_{\sigma_\mathcal{V}} \Phi_{\sigma_\mathcal{V},K}^\bu 
+\sum_{\sigma_{\mathcal{E}}\in K}\omega^{\bu,n,K}_{\sigma_\mathcal{E}}
\Phi^\rho_{\sigma_{\mathcal{E}},K}\bigg\}.
\end{equation*}
There is no reason to have a different value of $r_{\sigma_\mathcal{V}}^\bu$ unless for possible special needs, so we set $r_{\sigma_\mathcal{V}}^\bu=r^\bu$, and since a priori
$$\sum_{\sigma_\mathcal{V}\in K} \omega^{\rho,n+1}_{\sigma_\mathcal{V}}>0$$
 we get a unique value of $r^\bu$
 defined by
 \begin{equation}
\label{1}
\bigg (\sum_{\sigma_\mathcal{V}\in K} \omega^{\rho,n+1}_{\sigma_\mathcal{V}}\bigg ) r^\bu=\int_{\partial K} \bbf^\bm(U^{{n}})\cdot \bn\; d\gamma-\bigg \{\sum_{\sigma_\mathcal{V}\in K} \omega^{\rho,n+1}_{\sigma_\mathcal{V}} \Phi_{\sigma_\mathcal{V},K}^\bu 
+\sum_{\sigma_{\mathcal{E}}\in K}\omega^{\bu,n,K}_{\sigma_\mathcal{E}}
\Phi^\rho_{\sigma_{\mathcal{E}},K}\bigg\}.
\end{equation}

Once this is known, we can update the velocity and compute $\bu_{\sigma_{\mathcal{V}}}^{{n+1}}$.

 \medskip
\noindent \textbf{Energy.}\par 
Now we know $\rho^{{n}}$, $\rho^{{n+1}}$, $\bu^{{n}}$, $\bu^{{n+1}}$ and $e^{{n}}$, and have the \emph{updated} residuals for the velocity (there is no change for the density). Again we introduce a correction on the energy, $r_{\sigma_\mathcal{E}}^e$, and for the residual
$$\Psi_{\sigma_\mathcal{E}}^e=\Phi_{\sigma_\mathcal{E}}^e(U^{{n}})+r_{\sigma_\mathcal{E}}^e$$ to satisfy \eqref{energy}, we simply need:
\begin{equation}\label{2}
\sum_{\sigma_\mathcal{E}\in K}r_{\sigma_\mathcal{E}}^e=\int_{\partial K} \bbf^E(U^{{n}})\cdot \bn\; d\gamma-\bigg \{\sum_{\sigma_\mathcal{E}\in K}\Phi_{\sigma_\mathcal{E},K}^e+\sum_{\sigma_\mathcal{V}\in K} \theta_{\sigma_\mathcal{V}}^{\bm}\cdot \Phi_{\sigma_\mathcal{V},K}^\bu+\frac{1}{2}\sum_{\sigma_{\mathcal{E}}}\theta_{\sigma_\mathcal{V}}^{q^2, K}\Phi_{\sigma_{\mathcal{E}}}\bigg \}.
\end{equation}
 Since there is no reason to favour one degree of freedom with respect to the other ones, we take $r^e_{\sigma_{\mathcal{E}}}=r^e$, and again we can explicitly solve the equation and obtain the energy at the new time instance.

{
\subsection{Modifications for other schemes}\label{sec:dG}
We have presented this conservation recovery method using a class of schemes that might seem a bit narrow. In this section, we want to explain that it is not the case. This can apply to more general schemes, as soon as the update of any variable $w$ (density, velocity, energy), described by degrees of freedom $\sigma$ (point values, averages, moments), can be written as:
$$
\sum_{K, \sigma\in K} \Phi_\sigma^K.$$

\noindent \emph{About accuracy}:  The calculations made for the first order in time can be immediately extended to the higher accuracy ones in time.
One just has to add a temporal contribution into the new residuals, see for example \cite{zbMATH06826581} for more details. In Appendix \ref{appendix_dec} we briefly present a straightforward choice to obtain a high order accurate approximation, the Deferred Correction (DeC) approach which was used for the numerical results in Section \ref{results}.

 We also see that the exact form of the residuals is never used, so this can also be extended to any type of residuals, including for high order ones as in \cite{zbMATH06826581} or \cite{zbMATH06361106}. We also note that we have never used the global continuity of the velocity: instead of using $\big (W_h\big)^2$ for the velocity, we could have used $\big (V_h\big )^2$ in a discontinuous Galerkin like spirit.

{The coefficients of \eqref{coeff:corr} can be computed with any quadrature formula provided that the geometrical location of the quadrature points needed to evaluate the boundary integrals depend only on the faces and not the element, so that the edge contribution will sum up to zero. However, in the calculations we have always used enough points so that the quadrature formula are exact for the polynomial degree that we need. It is only to test global conservation that we have also used non exact quadrature formula. }

}
\section{Some numerical results}\label{numerics}

In this section, we want to illustrate the previous results and show that the method is effective. We are not claiming that these are the optimal ones, they can be seen more as a proof of concept. It is enough to describe what is done on $K=K_{j+1/2}$, for $j\in \Z$.
\subsection{Actual schemes}\label{sec:scheme}
In the following, $\hbbf_{j+1/2}$ is a numerical flux evaluated between the states 
\begin{equation*}
\begin{split}U_{j+1/2}^+&=\lim\limits_{x\rightarrow x_{j+1/2}, x>x_{j+1/2}}\big ( \rho, \rho\bu, e+\frac{1}{2}\rho\bu^2)(x)\\
 U_{j+1/2}^-&=\lim\limits_{x\rightarrow x_{j+1/2},x<x_{j+1/2}}( \rho, \rho\bu, e+\frac{1}{2}\rho\bu^2)(x).
 \end{split}
 \end{equation*}
 Here $\rho(x), \bu(x)$ and $e(x)$ are obtained from the approximation space. The flux $\hbbf$ has a $\rho$ component, a $\bm-$ component, and a total energy component, they are denoted by $\hbbf^\rho$, $\hbbf^m$, and $\hbbf^E$. Note that $\bu$ is continuous. In the numerical experiments, we will consider a solver constructed from an approximate Rieman solver because it appears we need intermediate states{ in the present description of the residual, see above: this allows to couple the neighbouring cells.} This could be the exact solver, we have used the HLLC one. Both provide solutions with the same success, the HLLC is easier to generalise.
 
We approximate the thermodynamic variables by polynomials of degree $r$ in each interval $K_{j+1/2}$  and the velocity by a continuous approximation which is a polynomial of degree $r+1$ in each interval $K_{j+1/2}$. We denote the approximation by $K(r+1)T(r)$. For the time discretisation, we use the DeC formulation briefly explained in Appendix \ref{appendix_dec}. For that reason, in each interval we expand the thermodynamic and kinetic function using B\'ezier polynomials since the integrals of the basis functions are always positive.

For simplicity, we reduce the formal time accuracy to second order, and we only need to describe the spatial terms: $\Phi_{\sigma_{\mathcal{E}}}^\rho$ for the density, $\Phi_{\sigma_{\mathcal{E}}}^e$  for the energy and $\Phi_{\sigma_{\mathcal{V}}}^\bu$  for the velocity.
The update of the density is done by the dG scheme:
\begin{equation}
\label{numer:rho}
\Phi_{\sigma_\mathcal{E}}^\rho=-\int_{K_{j+1/2}} \nabla\varphi_{\sigma_\mathcal{E}} \bbf^\rho\; d\bx +\bigg ( \hbbf_{j+1/2}^\rho \varphi_{\sigma_\mathcal{E}}(x_{j+1/2})-\hbbf_{j-1/2}^\rho \varphi_{\sigma_\mathcal{E}}(x_{j-1/2}) \bigg ).
\end{equation}

{\color{black}The update of the velocity is done by considering the centered residual:
\begin{equation}\label{numer:u:centered}
\rho^\star_K\; \Psi_{\sigma_\mathcal{V}}^\bu={\int_{K}\varphi_{\sigma_\mathcal{V}} \rho \bu\dpar{\bu}{x}\; d\bx-\int_K p\dpar{\varphi_{\sigma_\mathcal{V}}}{x}\;d\bx+\int_{\partial K} p^\star
\varphi_{\sigma_\mathcal{V}}\; d\gamma}{},
\end{equation} where $p^\star$ is the pressure evaluated at the quadrature points by the Riemann solver (this is why we have chosen HLLC), and $\rho^\star_K$ is the average of the density in $K$.
 Inspired by what is done in the RD context, and by \cite{burman}, we may need to consider a jump term for the velocity only because it is globally continuous. We have taken
\begin{equation}
\label{numer:u:jump}J_\sigma^K=\theta_K \beta_K h_K^2 \int_{\partial K} \big [ \nabla \varphi_{\sigma_\V}\big ] \; \big [ \nabla \bu\big ] \; d\gamma.\end{equation}
$\theta_K$  is a parameter (set to $0.1$ in the experiments), $\beta_K$ is an upper bound of the wave speeds in $K$
and a Local Lax Friedrich dissipation term
\label{numer:u}
\begin{equation}\label{numer:u:LxF}
D_\sigma^K=\alpha_K \big (\bu_{\sigma_\mathcal{V}}-\overline{\bu}\big )
\end{equation}
where $\alpha_K$ is an upper bound of the wave speeds in $K$ and $\overline{\bu}$ is the arithmetic average of the velocity within $K$. 
In practice, we will take either the residual 
\begin{subequations}
\begin{equation}\label{numer:u:1}
\Psi_{\sigma_\mathcal{V}}^\bu=\Psi_{\sigma_\mathcal{V}}^\bu+D_\sigma^K
\end{equation}
which will leads to a first order scheme
or
\begin{equation}\label{numer:u:2}
\Psi_{\sigma_\mathcal{V}}^\bu=\Psi_{\sigma_\mathcal{V}}^\bu+J_\sigma^K
\end{equation}
\end{subequations}
which will be higher order and stable.
}

{\color{black}
The update of the internal energy is simply done by
\begin{equation}\label{numer:e}
\Phi_{\sigma_{\mathcal{E}}}^e=\int_K\varphi_{\sigma_\mathcal{E}}\bigg (\bu\cdot \nabla e+(e+p)\text{ div }\bu\bigg )\;d\bx
\end{equation}
}

The schemes, even with the Euler forward time stepping, have no chance to be positivity preserving, and we note that the update of the velocity will be at most first order in time. Hence, inspired by the Residual Distribution schemes, we upgrade formal accuracy in two possible ways:
\begin{itemize}
\item Procedure 1: We use the residuals \eqref{numer:rho} and \eqref{numer:e} for the thermodynamic variables, and for the velocity, we replace $\Phi_{\sigma_\mathcal{V}}^\bu$ by $\big (\Phi_{\sigma_\mathcal{V}}^\bu\big )^\star$ defined by:
\begin{enumerate}
\item Compute $\Phi^{\bu}=\sum\limits_{\sigma_{\mathcal{V}}}\Phi_{\sigma_\mathcal{V}}^\bu$.
\item If $\Vert \Phi^{\bu}\Vert>0$, define 
$$x_{\sigma_\mathcal{V}}=\max\bigg (\dfrac{\Phi_{\sigma_\mathcal{V}}^\bu}{\Phi^{\bu}}, 0\bigg )$$
and
$$\big (\Phi_{\sigma_\mathcal{V}}^\bu\big )^\star=\dfrac{x_{\sigma_\mathcal{V}}}{\sum\limits_{\sigma_\mathcal{V}\in K_{j+1/2}} x_{\sigma_\mathcal{V}}}
\Phi^{\bu}.$$
\item Else $\big (\Phi_{\sigma_\mathcal{V}}^\bu\big )^\star=0$.
\end{enumerate}
\item Procedure 2: We do the same as before for $\Phi_{\sigma_\mathcal{V}}^\bu$, $\Phi_{\sigma_\mathcal{E}}^\rho$ and $\Phi_{\sigma_\mathcal{E}}^e$, where the thermodynamic residuals are now:
$$\Phi_{\sigma_\mathcal{E}}^\rho+\alpha_K \big ( \rho_{\sigma_\mathcal{E}}-\bar \rho_K\big )
\quad\text{and}\quad
\Phi_{\sigma_\mathcal{E}}^e+\alpha_K \big ( e_{\sigma_\mathcal{E}}-\bar e_K\big ),$$
where $\bar \rho_K$ (resp. $\bar e_K$) are the arithmetic average of the density DOFS  (resp. internal energy) in $K$.
\end{itemize}
{\color{black}We may also need to add a jump term of the type \eqref{numer:u:jump} on all the variables (we have not done this here.}
We refer the reader to \cite{zbMATH06826581} for more details, and in particular why formal accuracy is increased. This procedure will lead, in practice, to a scheme that preserves the positivity of the density and the pressure.

{\color{black} In the experiments where we want to test the accuracy of the method, we will use the combination \eqref{numer:rho},  \eqref{numer:u:2}, \eqref{numer:e}.}
\subsection{Results}\label{results}
\input{numerics}

\section{Conclusions}
In this paper, we have proposed a method to construct staggered high order schemes for compressible flows. The method has been illustrated on a staggered higher-order Resi\-dual Distribution (RD) scheme for compressible flow, but as explained in the text, this is not restricted to this particular class of schemes.The key elements are (i) a reformulation of the local conservation properties at the level of elements, and not faces as it is classically done, (ii) that the type stepping method is obtained by a combination of Euler forward steps, but this is more general than, for example SSP Runge Kutta: Defect correction methods can also be used. 

One of the contributions of this paper is to show how one can discretise a non-conservative version of the Euler equations of gas dynamics in Eulerian form and guarantee that the correct weak solutions are recovered. A series of classical problems considered in this paper show the accuracy and robustness of the proposed numerical scheme. The  scheme we have developed provides an accurate numerical approximation and the correction we have defined is effective. In addition, for solutions with shock, the scheme is parameter-free and does not require any artificial viscosity. 

Let us write a series of remarks to conclude this paper.
\begin{enumerate}
\item Though the numerical examples are mostly one-dimensional (because in this case one can compute the exact solution for comparison), the description of the correction introduced in Section \ref{howto}, as well as the conditions introduced in Proposition \ref{LxW} are formulated for general elements. 
\item The Residual Distribution formalism introduced here is not restrictive. In \cite{MR4090481}, it is shown that any classical scheme (Finite Volume, Finite Element, discontinuous Galerkin) can be rewritten equivalently in distribution form. If one approximates (for example) the velocity equation with another method, it is certainly possible to write the contribution at element level, as here, and then to rewrite the scheme in the semi-discrete form \eqref{semi_forw_Euler} (if first-order accuracy in time is chosen), 
or more general for higher in time approximation. Then, the key fact is to write the local conservation property, not at the level of faces between elements, but on the elements themselves: this is what is behind the proof of Proposition \ref{LxW}, thus  corrections  of the form \eqref{1} and \eqref{2} can be written. What is not guaranteed is that the modified scheme will still be stable. In all our experience, we have not see any degradation of the stability condition. We have used this type of  correction in other context, see e.g \cite{entropy,paola,abgrall2021reinterpretation}, and the conclusions are the same. This is however not a proof.
\end{enumerate}
Further investigations of high order Residual Distribution schemes and applications to different mathematical models will be considered in forthcoming works. 

\section*{Acknowledgements}
 I thank Dr. Bettina Wieber  for her very constructive comments. I am also grateful to Dr. Ksenya Ivanova  during her stay at I-Math for our discussions on this problem. \remi{I am also in debt with the two unknown reviewers that have led to drastic improvement with respect to the original submission.}

\appendix
\section{Short introduction to the Deferred Correction (DeC) approach}\label{appendix_dec}
We consider again a hyperbolic system in the form
\begin{equation}
\dfrac{\partial U}{\partial t} + \mathrm{L}(U) = 0
\end{equation} 
which we want to approximate with a high-order accurate scheme in time. To do so, we will use the Deferred Correction (DeC) approach. The aim of DeC schemes is to avoid implicit methods, without losing the high order of accuracy of a scheme. The high order method that we want to approximate will be denoted by $\mathcal{L}^2$. To use the DeC procedure, we also need another method, which is easy and fast to be solved with low order of accuracy which will be denoted by $\mathcal{L}^{1}$. The DeC algorithm is providing an iterative procedure that approximates the solution of the $\mathcal{L}^2$ scheme $U^*$ in the following way:

\begin{equation}\begin{array}{cc}
\mathcal{L}^1\left(U^{{n+1}}\right) = 0,
\end{array}
\end{equation} 

\begin{equation}\begin{array}{ccc}
\mathcal{L}^1\left(U^{(k)}\right) =\mathcal{L}^1\left(U^{(k-1)}\right) - \mathcal{L}^2\left(U^{(k-1)}\right) , & \text{with} & k = 2, .., K,
\end{array}\end{equation}
where $K$ is the number of iterations that we compute. We need as many iterations as the order of accuracy that we want to reach. We know from \cite{AbgrallTorlo2020}:

\begin{prop}\label{lemma}
Let $\mathcal{L}^1$ and $\mathcal{L}^2$ be two operators defined on $R^m$, which depend on the discretization scale $\Delta \sim \Delta x \sim \Delta t,$ such that
\begin{itemize}
\item $\mathcal{L}^1$  is coercive with respect to a norm, i.e., $\exists \alpha_1 > 0$ independent of $\Delta $, such that for any $U, V$ we have that 
$$\alpha_1 \norm{U - V} \le \norm{\mathcal{L}^1(U) - \mathcal{L}^1(V)},$$

\item $\mathcal{L}^1 - \mathcal{L}^2$ is Lipschitz with constant $\alpha_2 > 0$ uniformly with respect to $\Delta$, i.e., for any $U, V$
$$\norm{\left(\mathcal{L}^1(U) - \mathcal{L}^2(U)\right) - \left(\mathcal{L}^1(V) - \mathcal{L}^2(V)\right)} \le \alpha_2\Delta\norm{U- V}.$$ 

We also assume that there exists a unique $U^*_{\Delta}$ such that $\mathcal{L}^2(U^*_{\Delta}) = 0.$ Then, if $\eta := \dfrac{\alpha_2}{\alpha_1}\Delta < 1,$ the DeC is converging to $U^*_{\Delta}$ and after $k$ iterations the error $\norm{U^{(k)} - U^*_{\Delta}}$ is smaller than $\eta^k\norm{U^{{n}} - U^*_{\Delta}}$.
\end{itemize}
\end{prop}

\bigskip
Following the proceeding in Section \ref{discussion}, we get for a second-order DeC scheme:
\begin{equation*}
\Phi^\bu_{\sigma_\mathcal{V},K}=\frac{1}{2}\big ( \Phi^\bu_{\sigma_\mathcal{V},K}(U^{(k)})+\Phi^\bu_{\sigma_\mathcal{V},K}(U^{{n}})
\big ), \quad \Phi^\rho_{\sigma_{\mathcal{E}},K}=\frac{1}{2}\big ( \Phi^\rho_{\sigma_\mathcal{E},K}(U^{(k)})+\Phi^\rho_{\sigma_\mathcal{E},K}(U^{{n}})\big ).
\end{equation*}
The calculations can also be immediately extended to higher accuracy in time by modifying the above half sums.

\input{lw}

\bibliographystyle{unsrt}
\bibliography{paper}
\end{document}

%% file: numerics.tex
Here we solve a series of shock tube problems to assess the accuracy and robustness of the proposed RD staggered scheme. For the numerical experiments of this section, we will use the ideal EOS for gas, linking the pressure, the internal energy and the density: $p = (\gamma -1)\rho e,$ where $\gamma = 1.4$. {All the solutions are displayed with $1000$ points: all the examples are shock tube problems where the exact solution is known, so we can show the convergence of the method to the exact solution. This is not needed for other purposes, such as stability or other reasons. We have also computed the solution with a more reasonable number of points, the results are similar to what could be obtained with more classical methods when the correction is activated. We do not show them here to lower the number of plots.}

{\color{black}\subsubsection{A smooth case}\label{sec:smooth}
The purpose is to test accuracy. We test the accuracy of our scheme (with corrections) on a smooth isentropic flow problem similar to the test case introduced in \cite{ChengShu2014}. The initial data for our test problem is the following:
\begin{equation*}
\rho_0(x) = 1 + 0.9\sin(2 \pi x), \quad u_0(x) = 0, \quad p_0(x) = \rho^{\gamma}(x,0), \quad x \in [-1,1].
\end{equation*}
with polytropic index $\gamma=3$ and periodic boundary conditions. 

The exact density and velocity in this case can be obtained by the method of characteristics and is explicitly given by
\begin{equation*}
\rho(x,t) = \dfrac12\big( \rho_0(x_1) + \rho_0(x_2)\big), \quad u(x,t) = \sqrt{3}\big(\rho(x,t)-\rho_0(x_1) \big),
\end{equation*}
where for each coordinate $x$ and time $t$ the values $x_1$ and $x_2$ are solutions of the nonlinear equations
\begin{align*}
& x + \sqrt{3}\rho_0(x_1) t - x_1 = 0, \\
& x - \sqrt{3}\rho_0(x_2) t - x_2 = 0.
\end{align*}
The errors ($L^1$ only is plotted because all the others shows a similar behaviour) is displayed in figure \ref{errors}
\begin{figure}[h]
\subfigure[]{\includegraphics[width=0.45\textwidth]{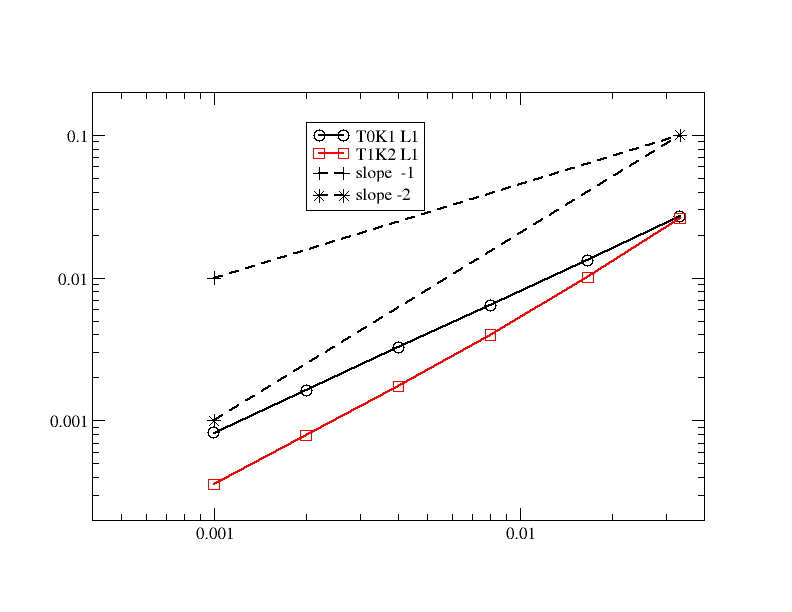}}
\subfigure[]{\includegraphics[width=0.45\textwidth]{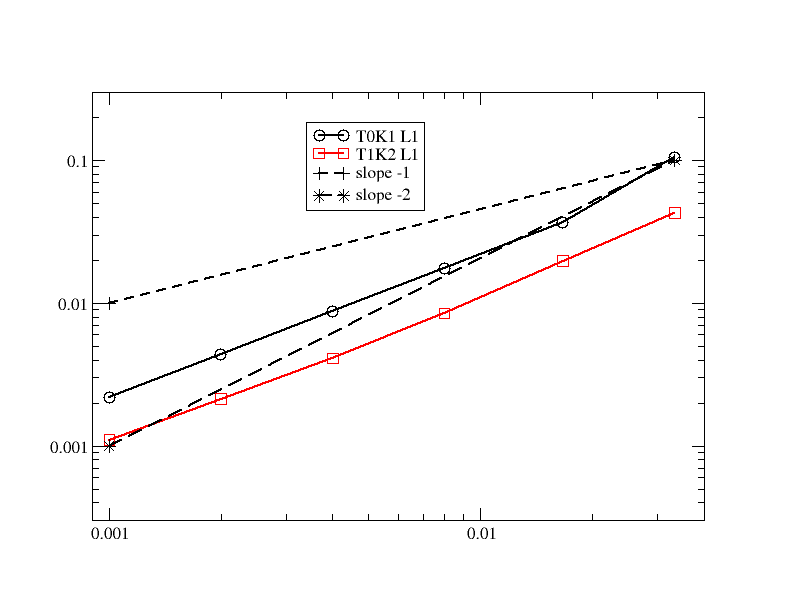}}
\caption{\label{errors} (a) error on the density, (b) error on the velocity}
\end{figure}
The errors are somehow between second and first order: the scheme in time is second order at most in this implementation, this is a choice, we could have used a third order scheme. }
\subsubsection{The Sod shock tube problem}
The Sod shock tube is a common one-dimensional Riemann problem for the illustration of the interesting behavior of numerical solutions to hyperbolic Euler equations of gas dynamics. The structure of the solution involves three distinct waves: a left rarefaction wave, a contact discontinuity, and a right shock wave. This test case is used to determine if a scheme recovers properly discrete Rankine-Hugoniot relations on the shock. If we put the initial discontinuity at $x=0.5$ in the domain $[0,1]$, the initial data for this problem is given as follows:

\begin{equation}\label{InCondSod}
(\rho_0, u_0, p_0) = \left\{\begin{array}{ccc}
(1.0,0.0,1.0), & \text{if}  & x<0.5,\\
\\
(0.125,0.0,0.1), & \text{if}& x>0.5.
\end{array}\right. 
\end{equation}
  In Figure \ref{sod}, profiles of density, velocity and pressure are depicted with a reference solution for a mesh containing $1000$ cells.  We also have plotted the solution obtained without any correction. Both have been obtained with the T0K1 scheme, and first order in time. We see that the uncorrected solution is completely off, as expected, but also that the correction we have defined provides an accurate approximation of all three distinct waves. 
\begin{figure}
\begin{center}
\subfigure[]{\includegraphics[width=0.45\textwidth]{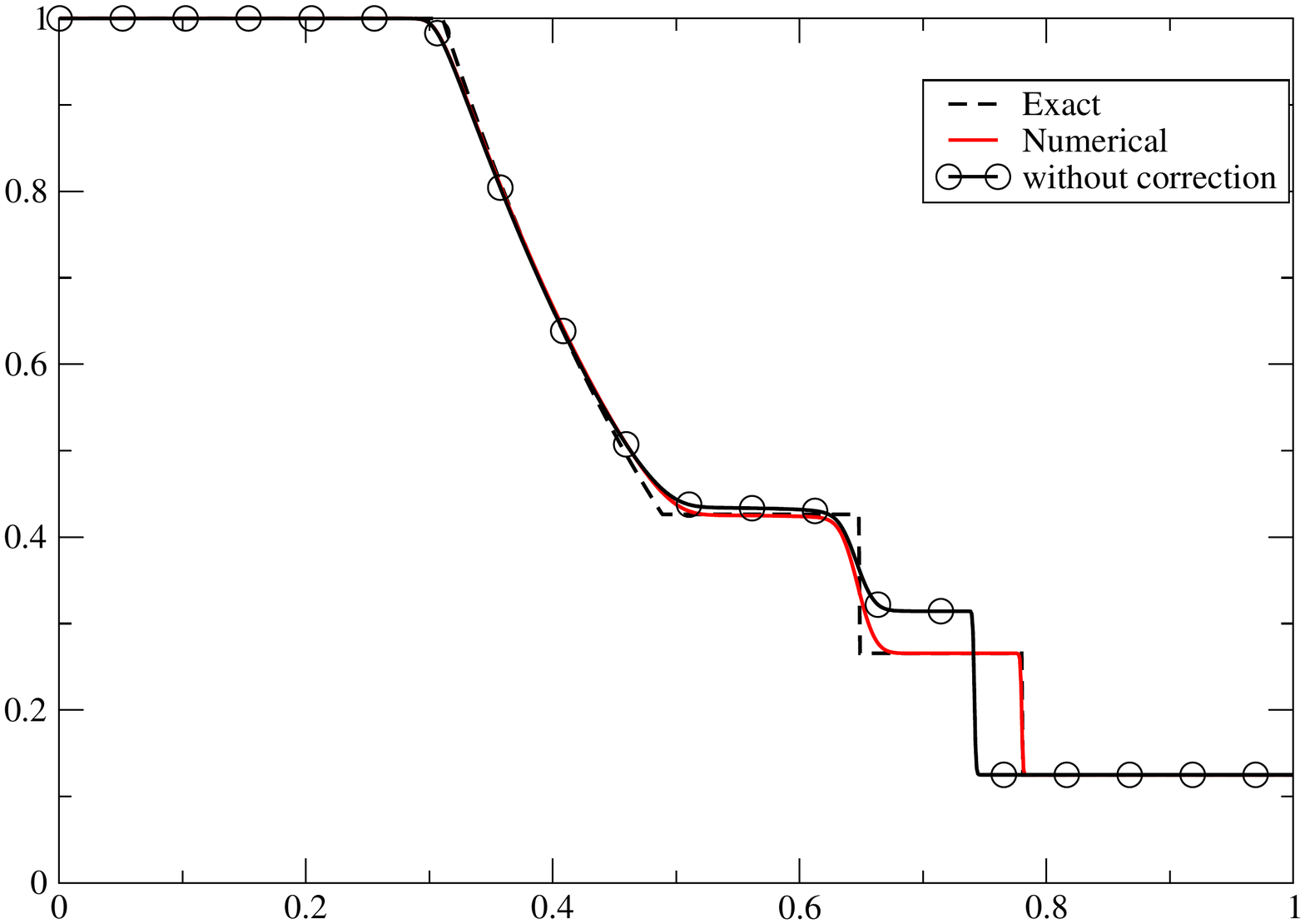}}
\subfigure[]{\includegraphics[width=0.45\textwidth]{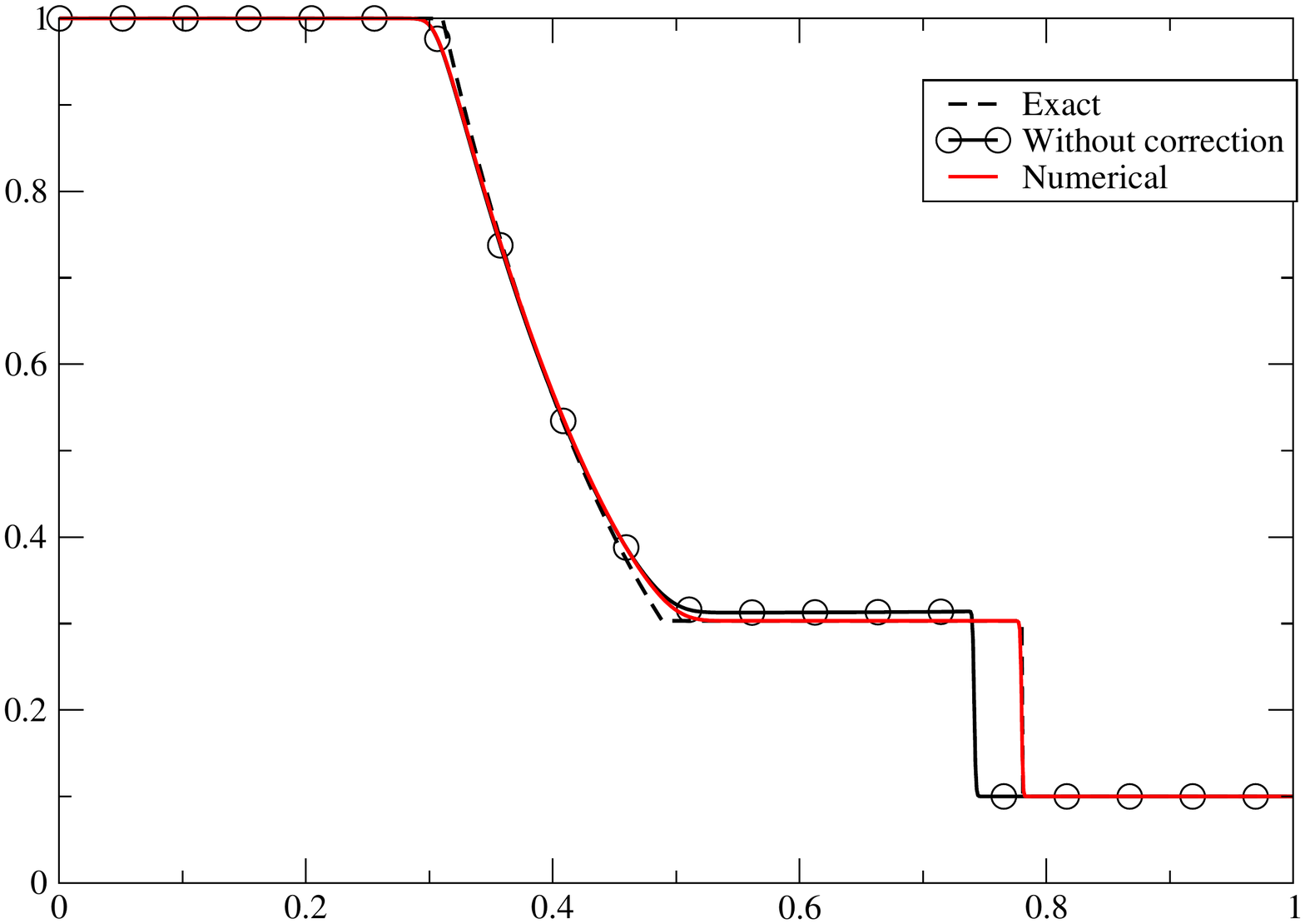}}
\subfigure[]{\includegraphics[width=0.45\textwidth]{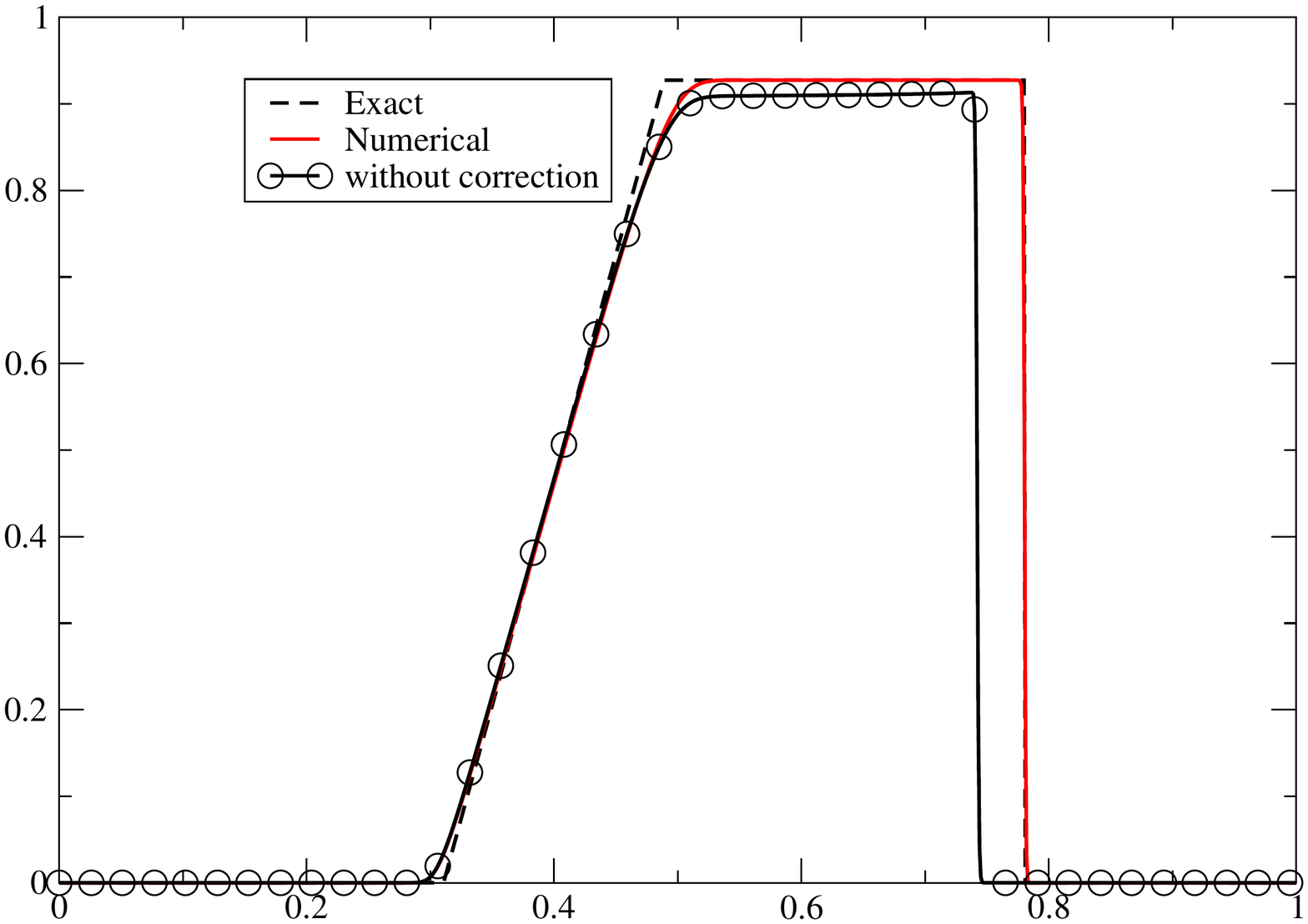}}
\end{center}
\caption{Solution of the Sod shock tube problem for (a) density, (b) velocity and (c) pressure at time $T = 0.16$ with $CFL = 0.4$. In each figure, the exact, numerical as well as the solution without correction are depicted.}
\label{sod}\end{figure}

In Figure \ref{sod2}, we show the results obtained by the first order (K1T0) and second order (K2T1) in time and space schemes. One can notice some improvements, but this example mostly shows that the correction is also effective when we use representations with polynomials of higher degree.
\begin{figure}
\begin{center}
\subfigure[]{\includegraphics[width=0.45\textwidth]{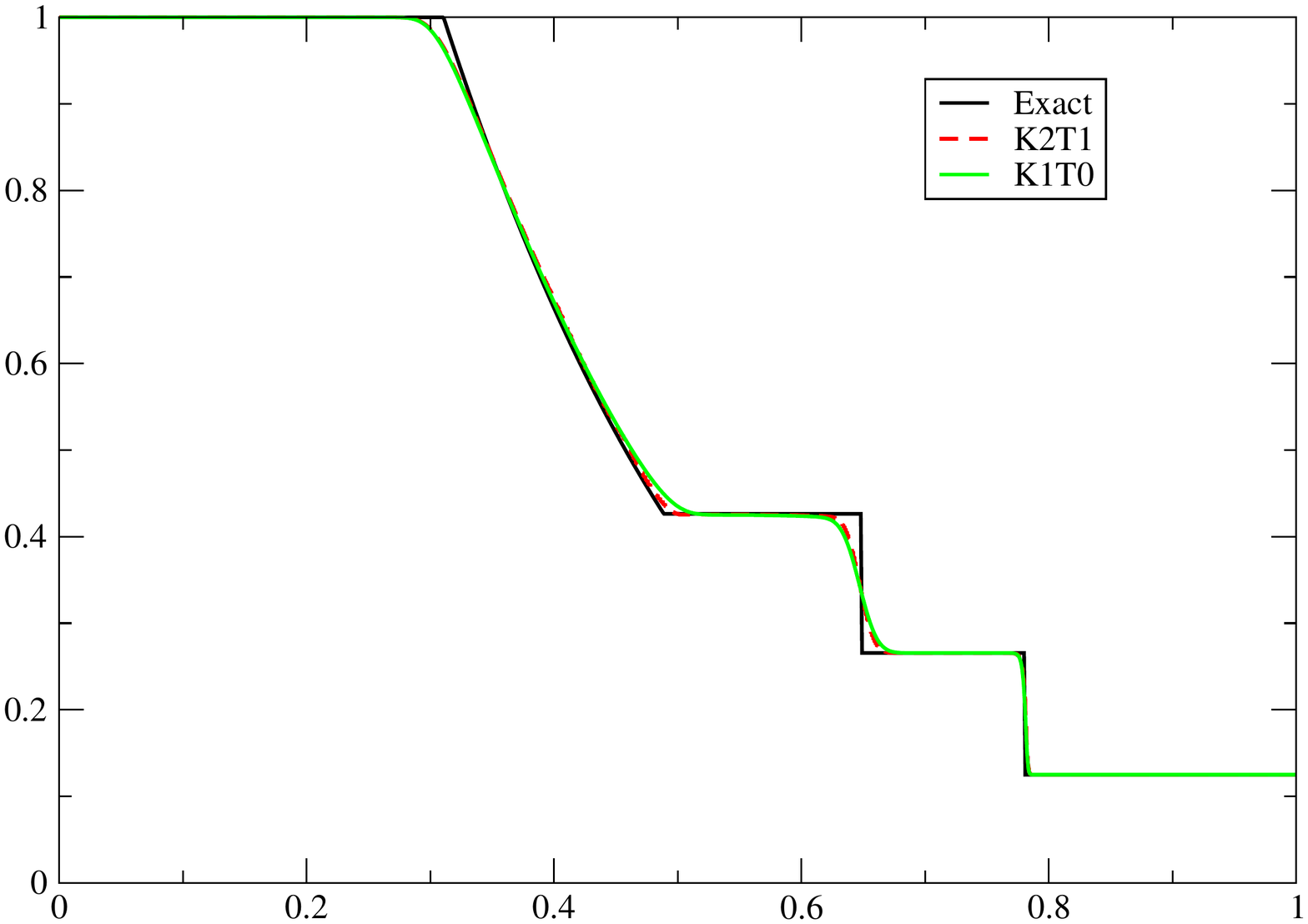}}
\subfigure[]{\includegraphics[width=0.45\textwidth]{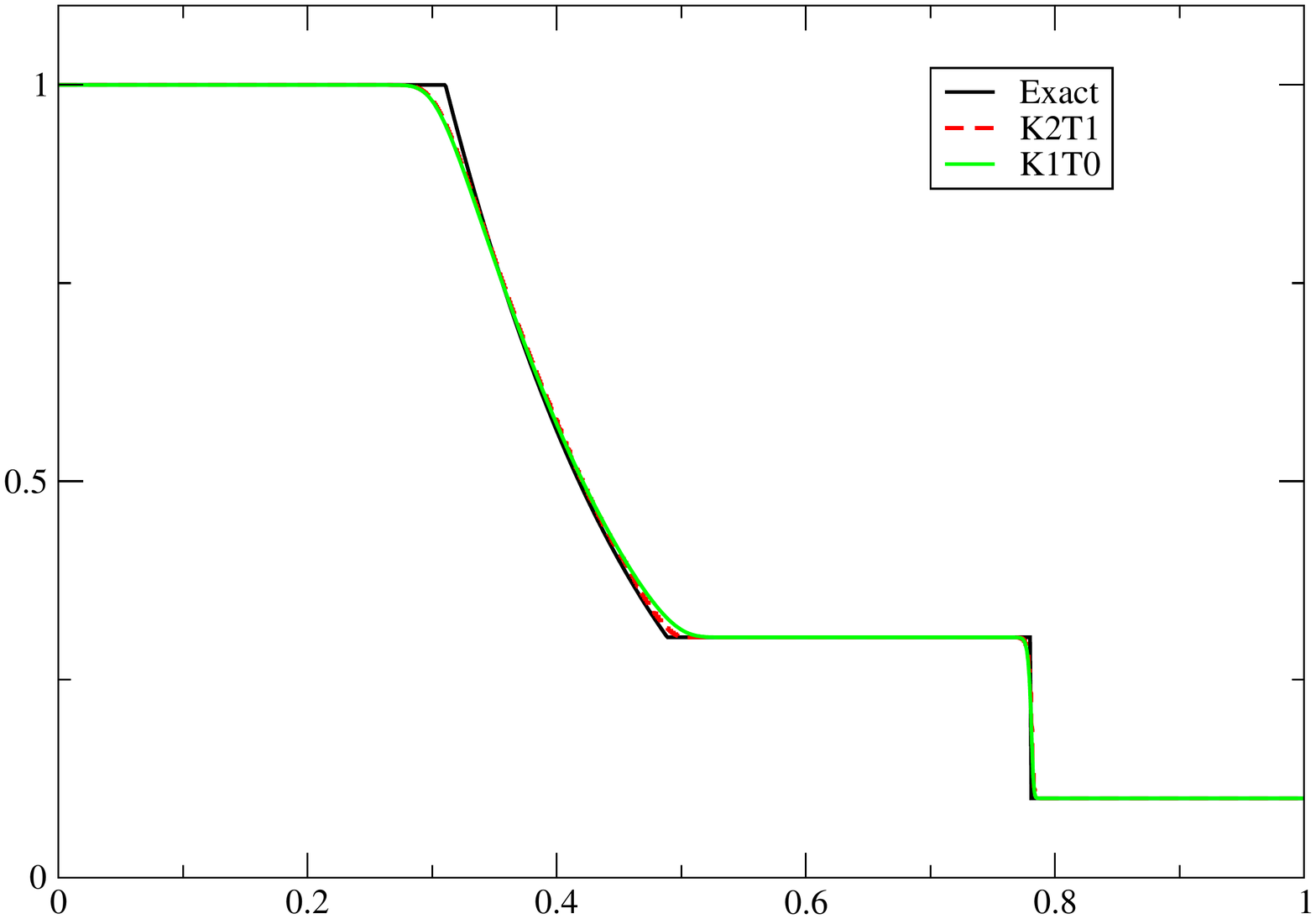}}
\subfigure[]{\includegraphics[width=0.45\textwidth]{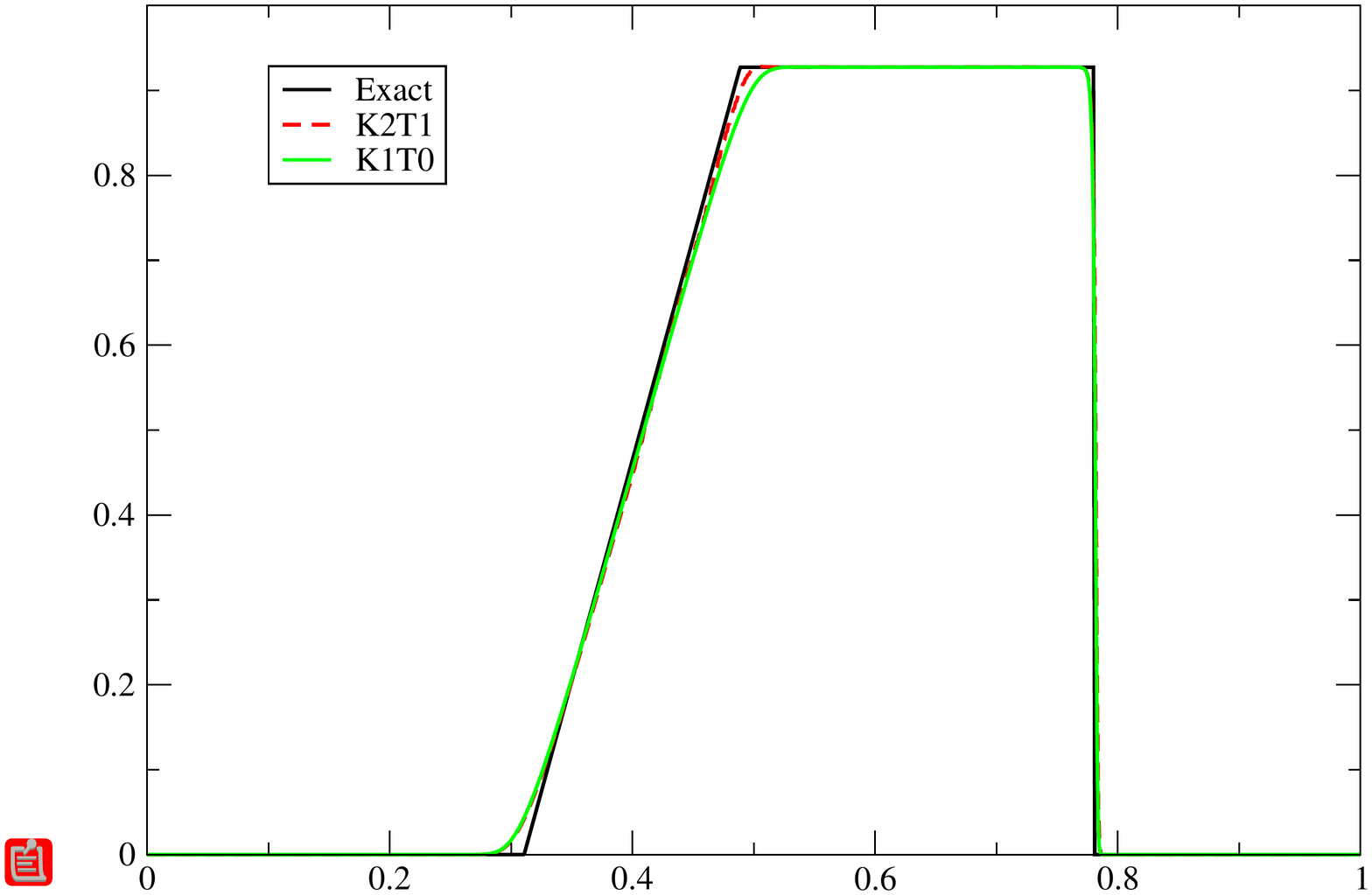}}
\end{center}
\caption{Solution of the Sod shock tube problem for (a) density, (b) velocity and (c) pressure at time $T = 0.16$ with $CFL = 0.4$. In each figure, the exact solution and the solutions obtained with K2T1 and K1T0 on a mesh with $1000$ points are depicted.}
\label{sod2}\end{figure}

\subsubsection{Strong shock}
The next test problem contains a left rarefaction wave, a contact discontinuity, and a strong right shock wave. This test case highlights the robustness of the numerical methods for fluid dynamics.
The initial data, again in the domain $[0,1]$,  are:
\begin{equation}\label{strong}
(\rho_0, u_0, p_0) = \left\{\begin{array}{ccc}
(1.0,0.0,1000.0), & \text{if}  & x<0.5,\\
\\
(1.0,0.0,0.01), & \text{if} & x>0.5.
\end{array}\right. 
\end{equation}
In Figure \ref{strong_shock}, profiles of density, velocity and pressure are depicted with a reference solution for a mesh containing $1000$ cells. It indicates that the first-order scheme can accurately resolve strong shocks. As before, the results of Figure \ref{strong_shock2} show that the correction is effective by comparing the results of the first order (K1T0) and second order (K2T1) in time and space schemes.

\begin{figure}
\begin{center}
\subfigure[]{\includegraphics[width=0.45\textwidth]{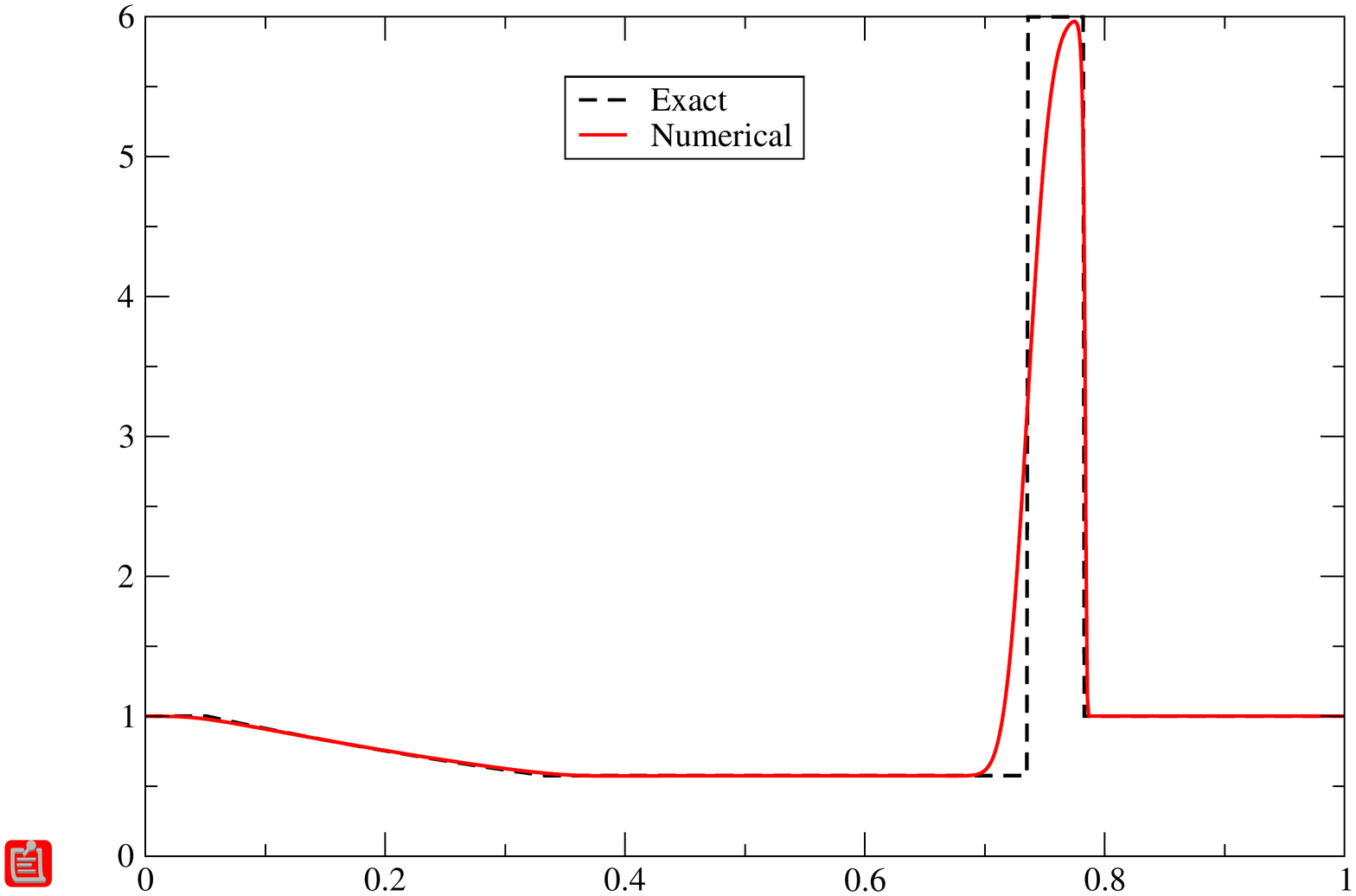}}
\subfigure[]{\includegraphics[width=0.45\textwidth]{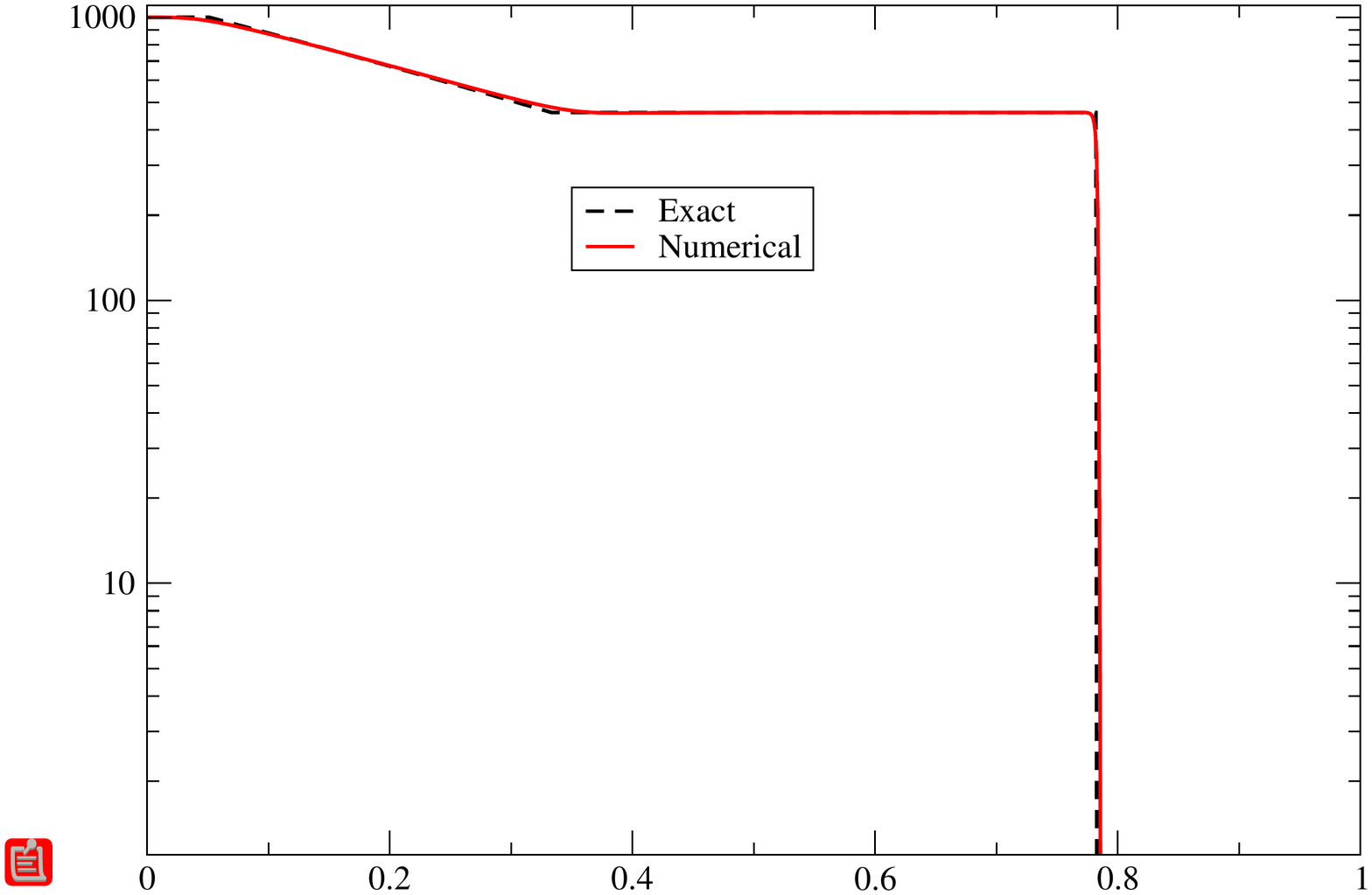}}
\subfigure[]{\includegraphics[width=0.45\textwidth]{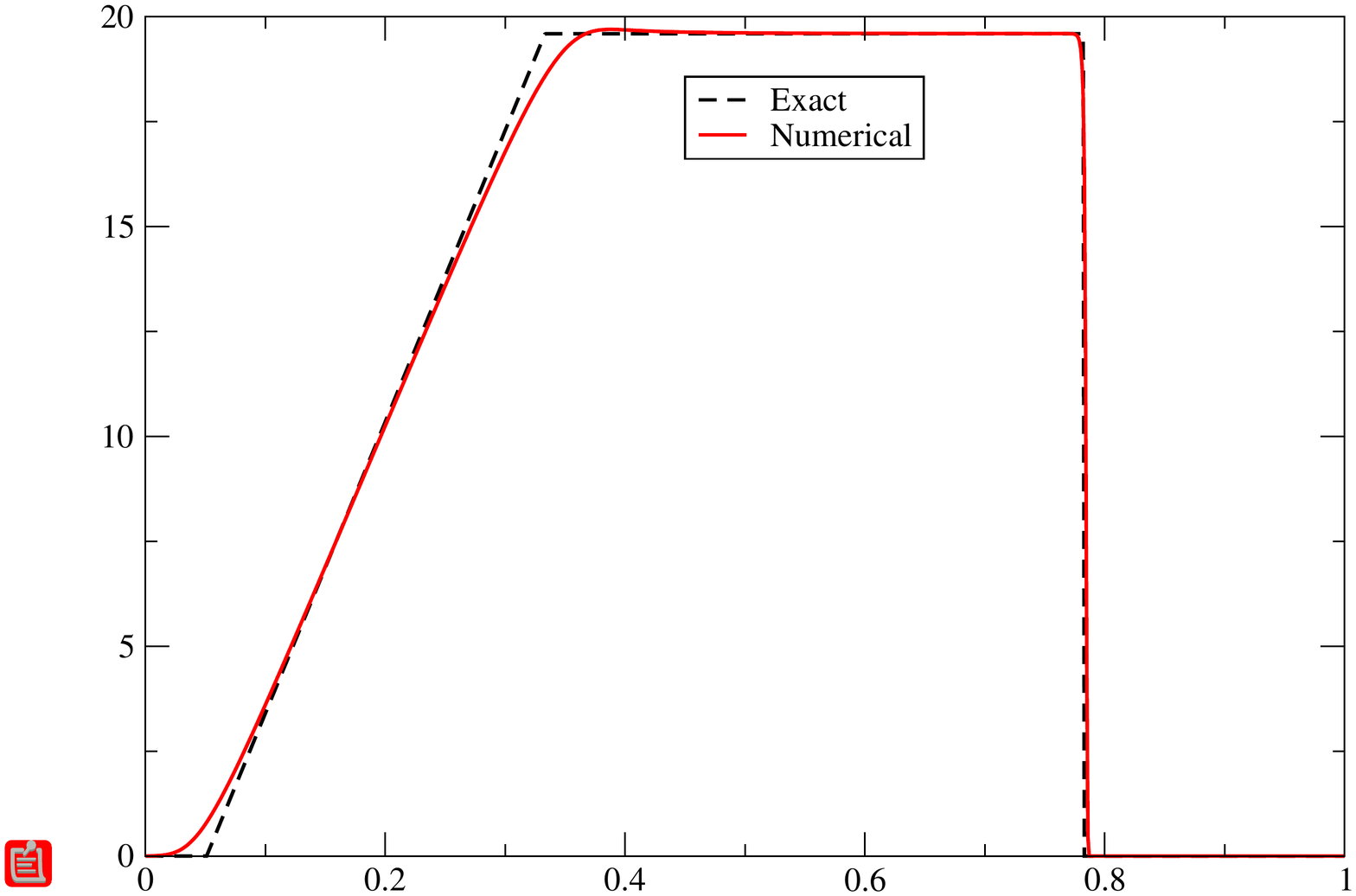}}
\end{center}
\caption{Exact and numerical solutions of the strong shock problem for (a) density, (b) velocity and (c) pressure at time $T = 0.012$ with $CFL = 0.4$.}
\label{strong_shock}\end{figure}

\begin{figure}
\begin{center}
\subfigure[]{\includegraphics[width=0.45\textwidth]{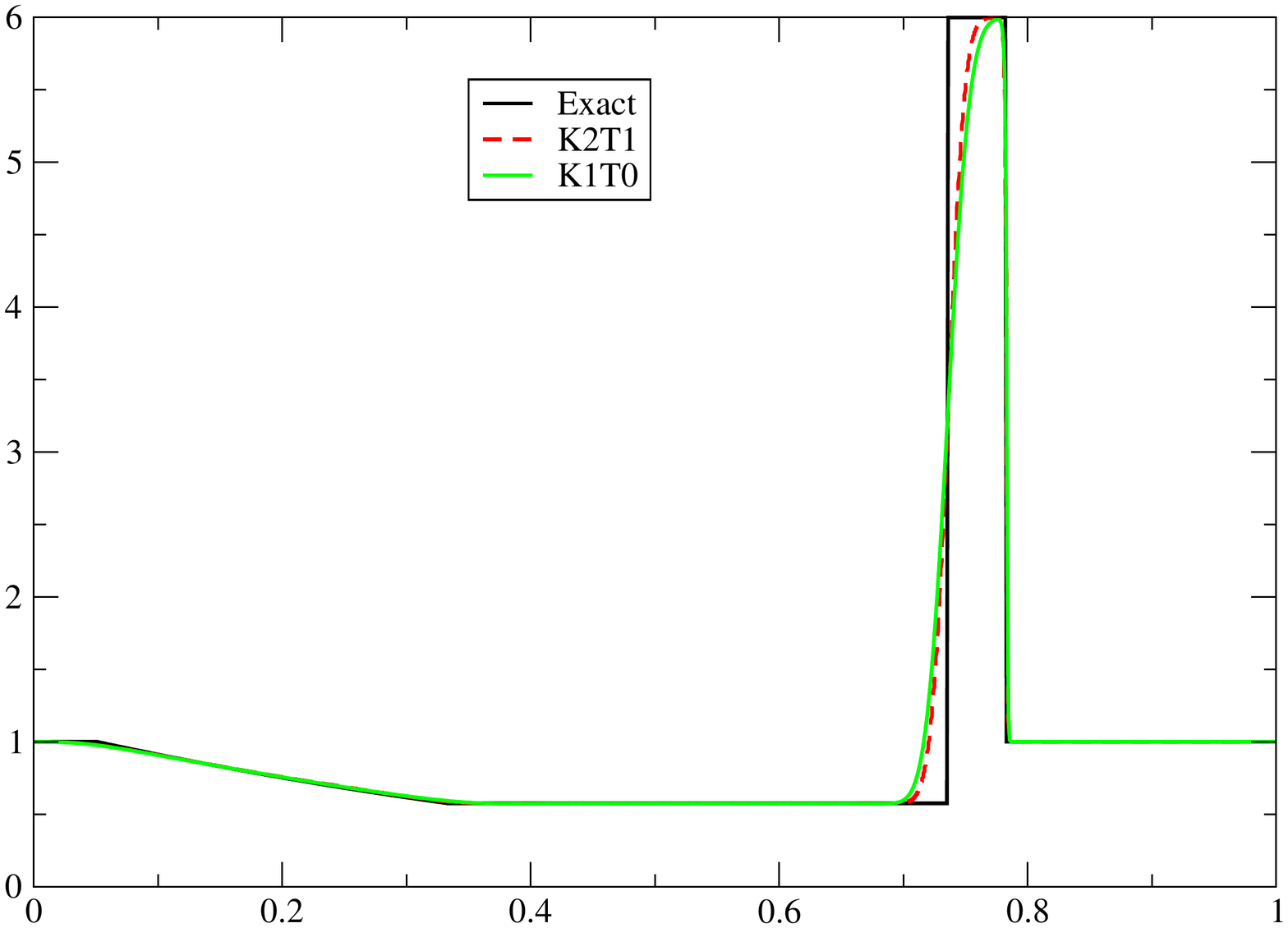}}
\subfigure[]{\includegraphics[width=0.45\textwidth]{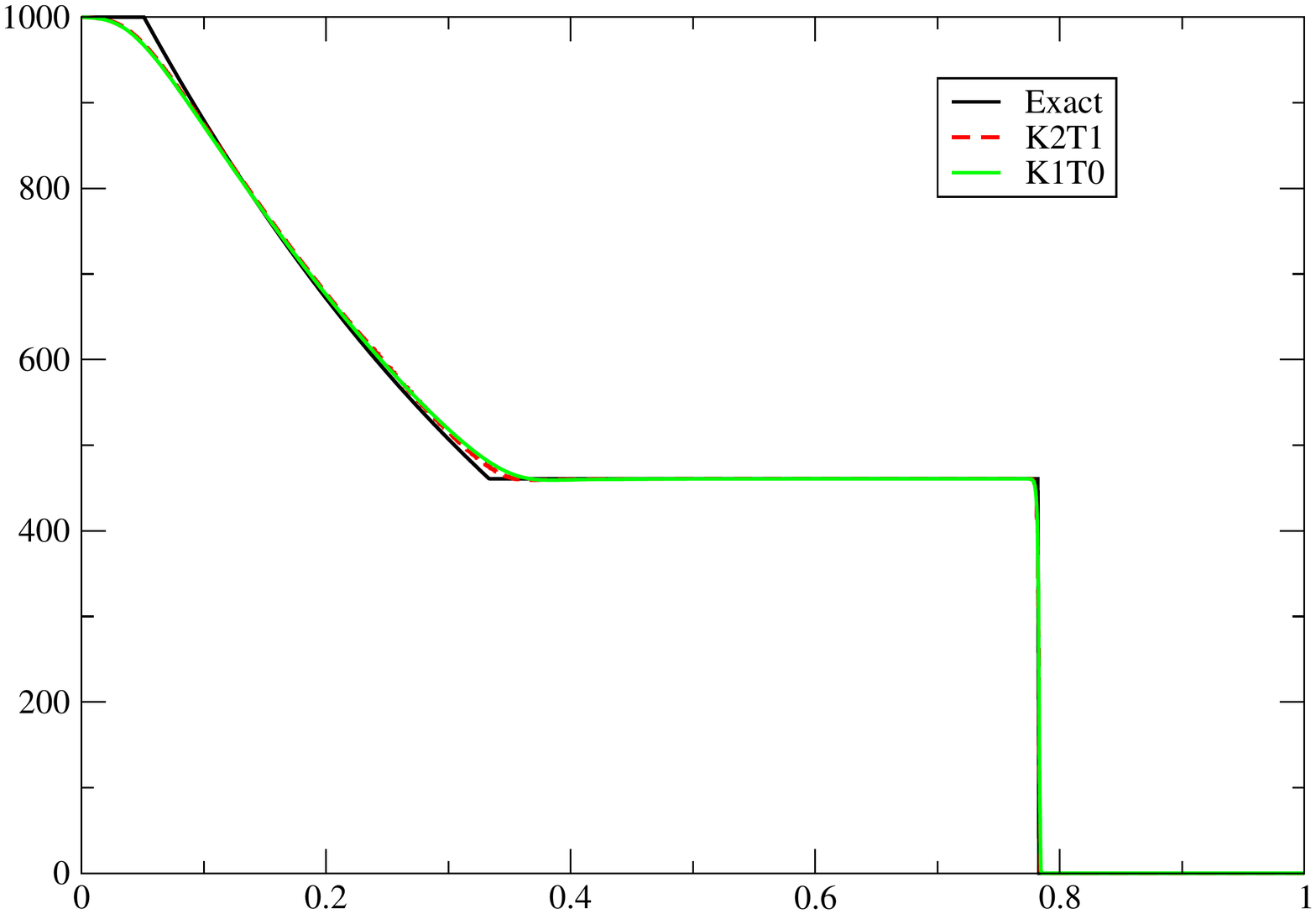}}
\subfigure[]{\includegraphics[width=0.45\textwidth]{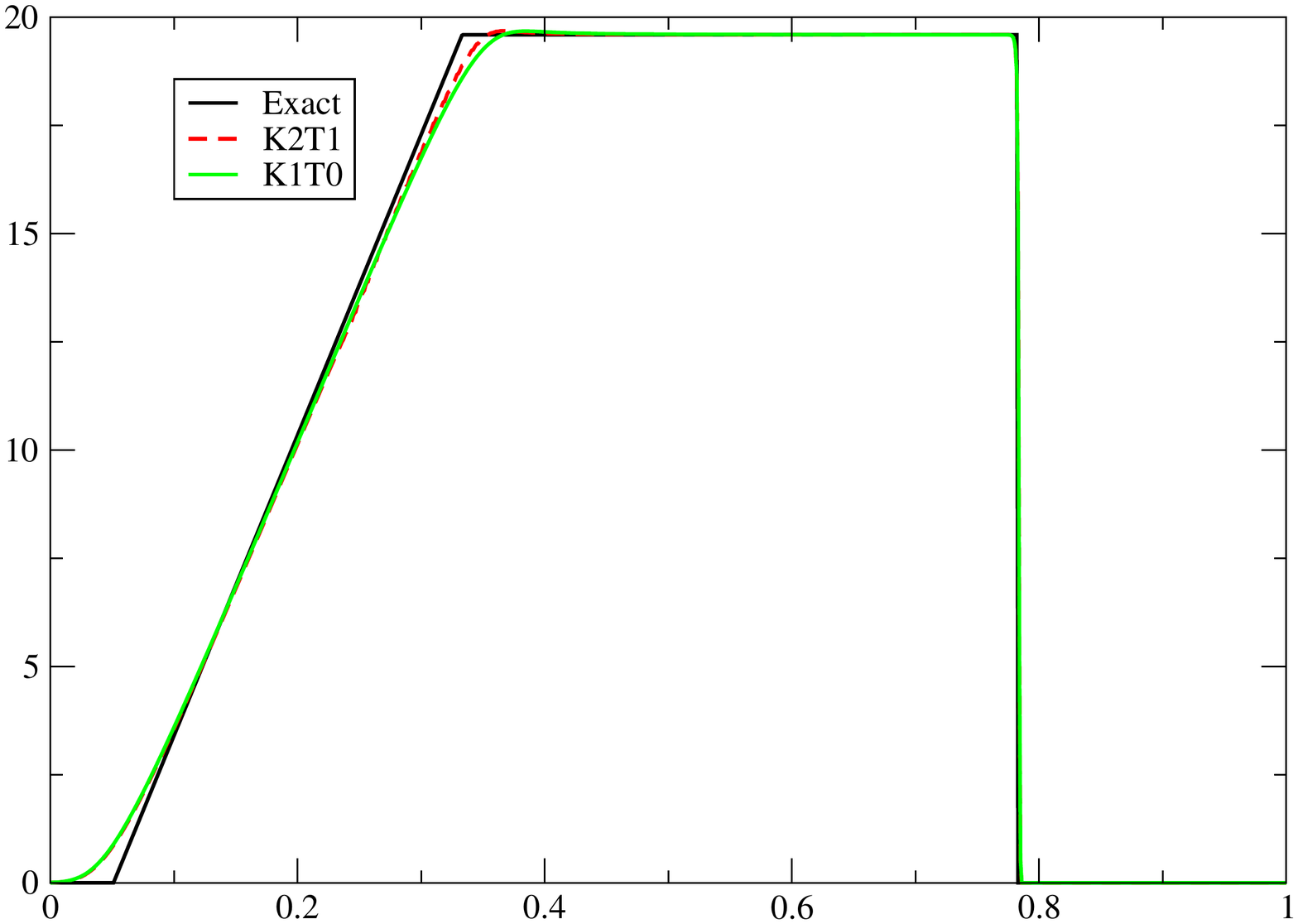}}
\end{center}
\caption{Solution of the strong shock problem for (a) density, (b) velocity and (c) pressure at time $T = 0.012$ with $CFL = 0.4$. In each figure, the exact solution and the solutions obtained with K2T1 and K1T0 are depicted.}
\label{strong_shock2}\end{figure}

\subsubsection{123-problem}
For the next test, called the 123 problem, the solution consists of a left rarefaction wave, a contact discontinuity and a right rarefaction wave. Two rarefaction waves are traveling in opposite directions. A low-density and low-pressure region is generated in between. 
The initial data for this problem is given as follows:
\begin{equation}\label{123}
(\rho_0, u_0, p_0) = \left\{\begin{array}{ccc}
(1.0,-2.0,0.4), & \text{if}  & 0.0\le x<0.5,
\\
(1.0,2.0,0.4), & \text{if} & 0.5<x<1.
\end{array}\right. 
\end{equation}
The results for the first-order scheme are depicted
in Figure \ref{123problem} with a reference solution on a mesh containing $1000$ cells. In Figure \ref{123problem2} we show again that the correction is effective by comparing the results of the first order (K1T0) and second order (K2T1) in time and space schemes.

\begin{figure}
\begin{center}
\subfigure[]{\includegraphics[width=0.45\textwidth]{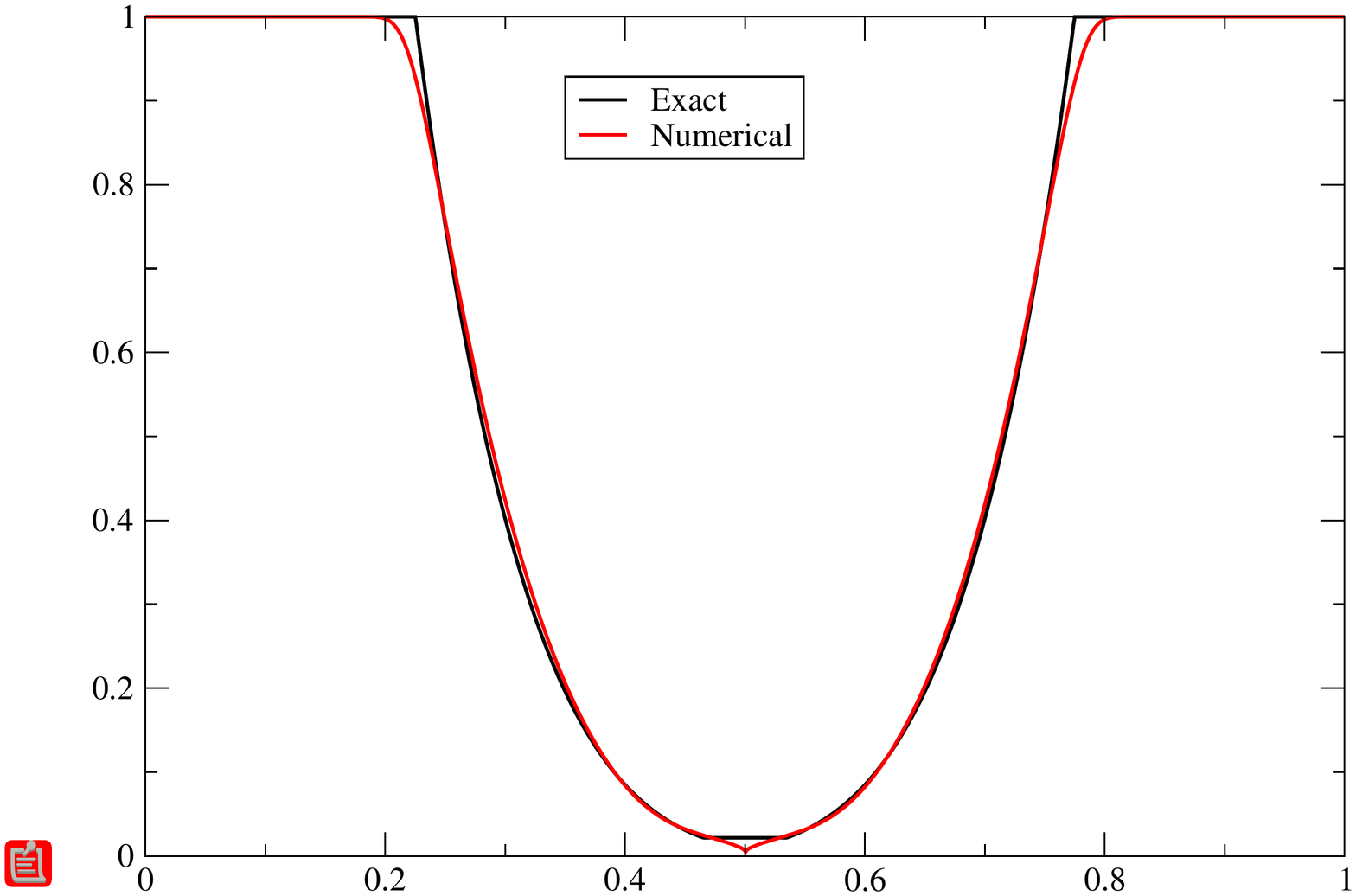}}
\subfigure[]{\includegraphics[width=0.45\textwidth]{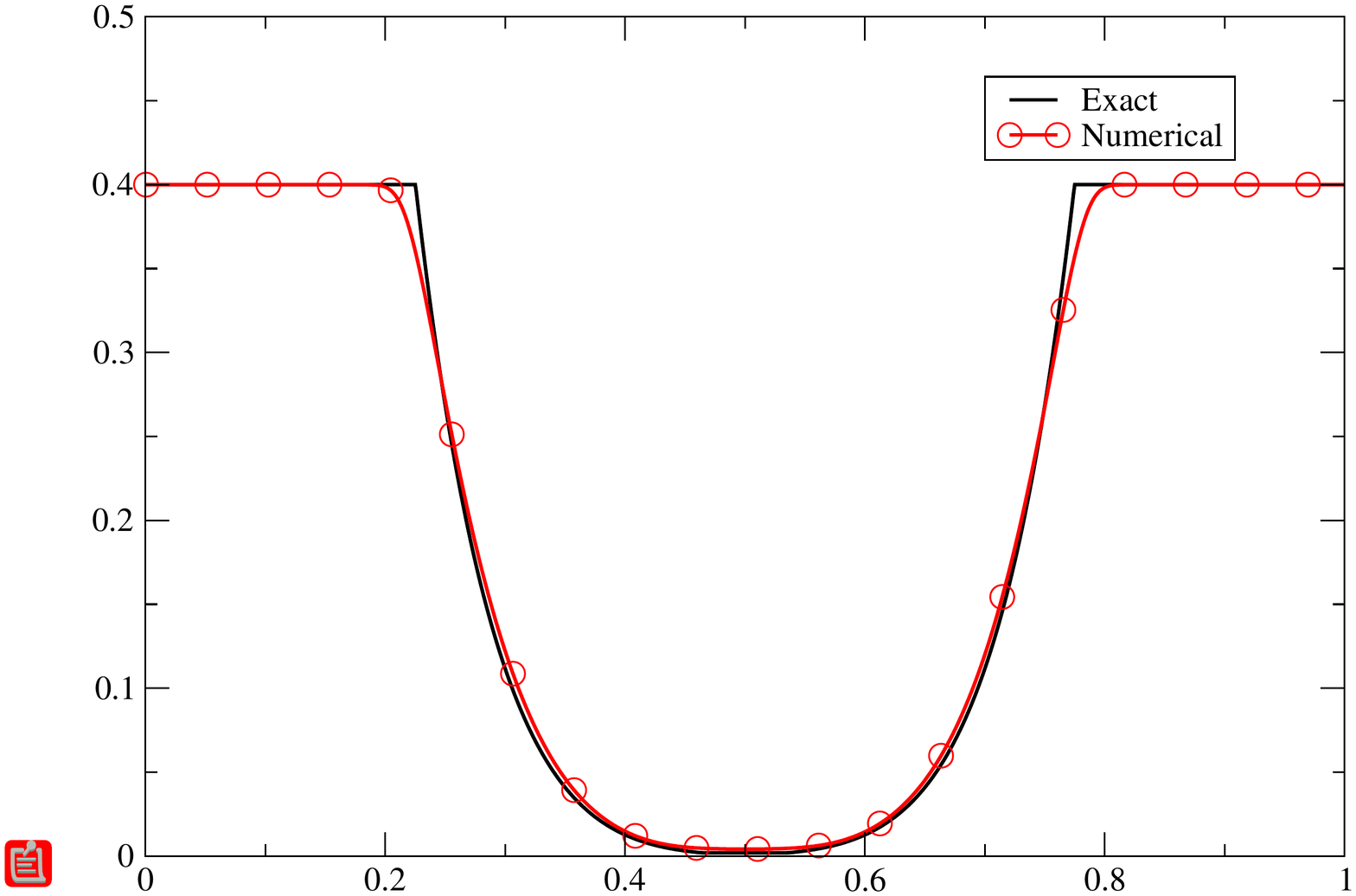}}
\subfigure[]{\includegraphics[width=0.45\textwidth]{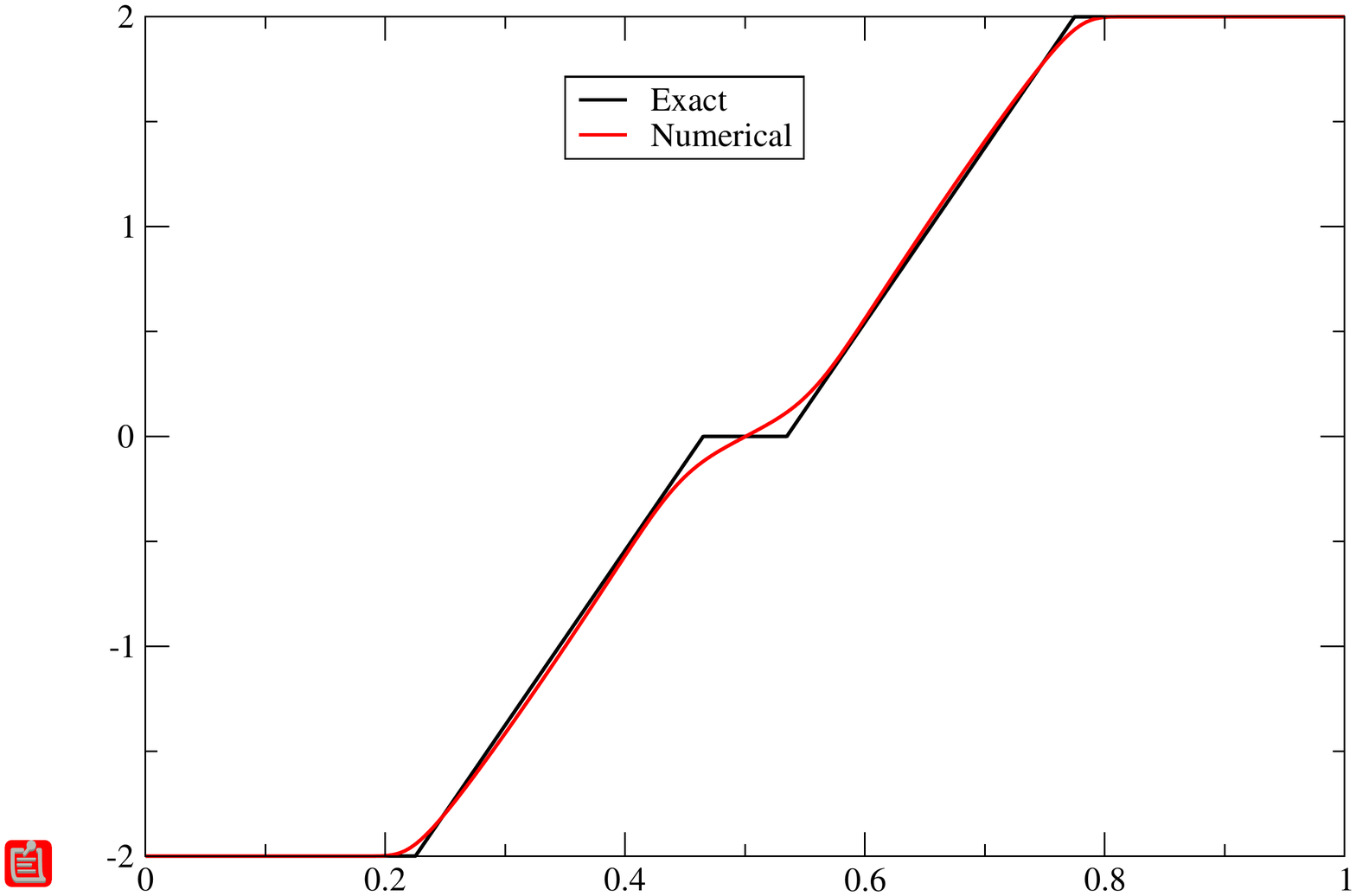}}
\end{center}

\caption{Exact and numerical solutions of the $123$-problem for (a) density, (b) velocity and (c) pressure at time $T = 0.15$ with $CFL = 0.4$.}
\label{123problem}\end{figure}
\begin{figure}
\begin{center}
\subfigure[]{\includegraphics[width=0.45\textwidth]{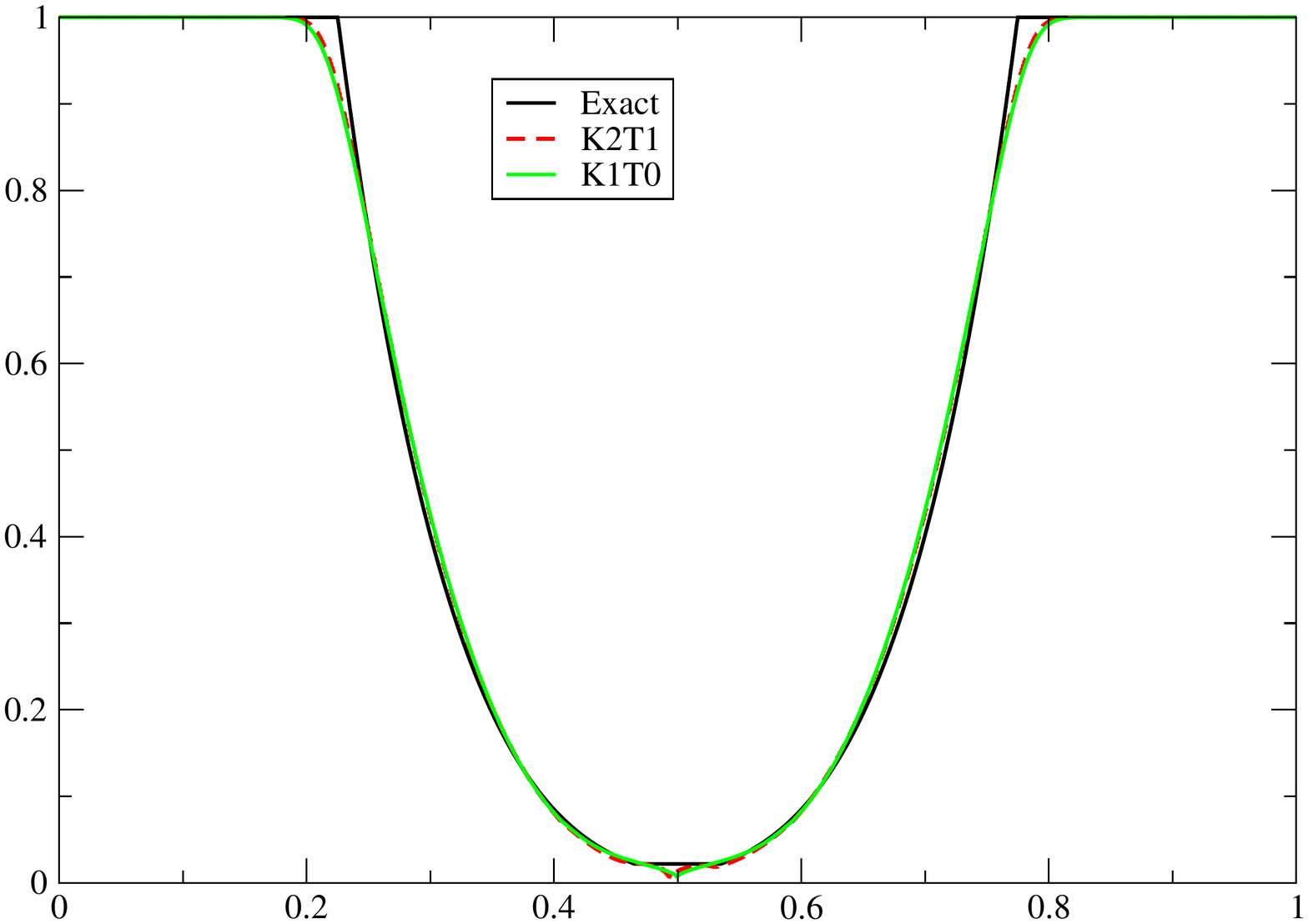}}
\subfigure[]{\includegraphics[width=0.45\textwidth]{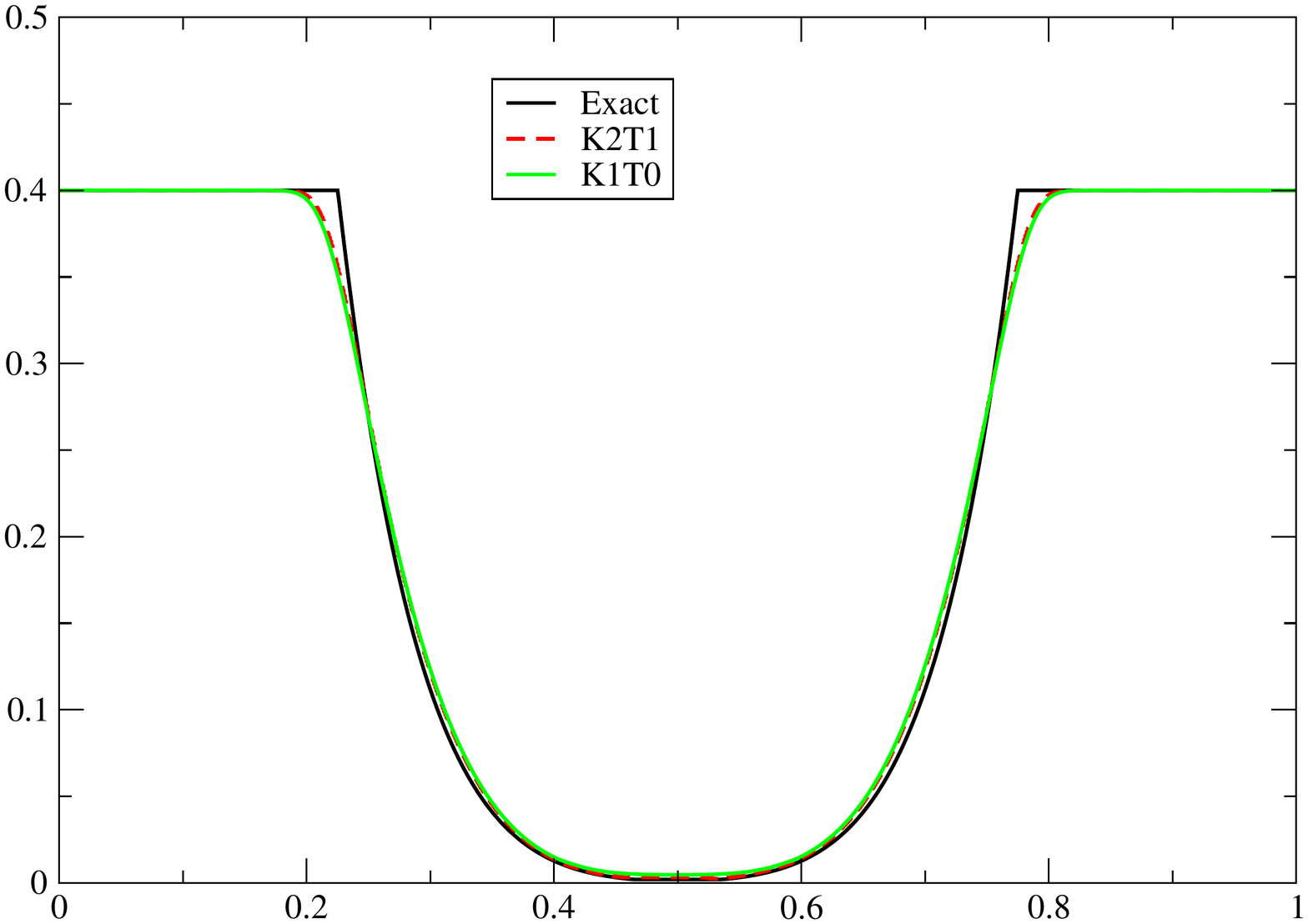}}
\subfigure[]{\includegraphics[width=0.45\textwidth]{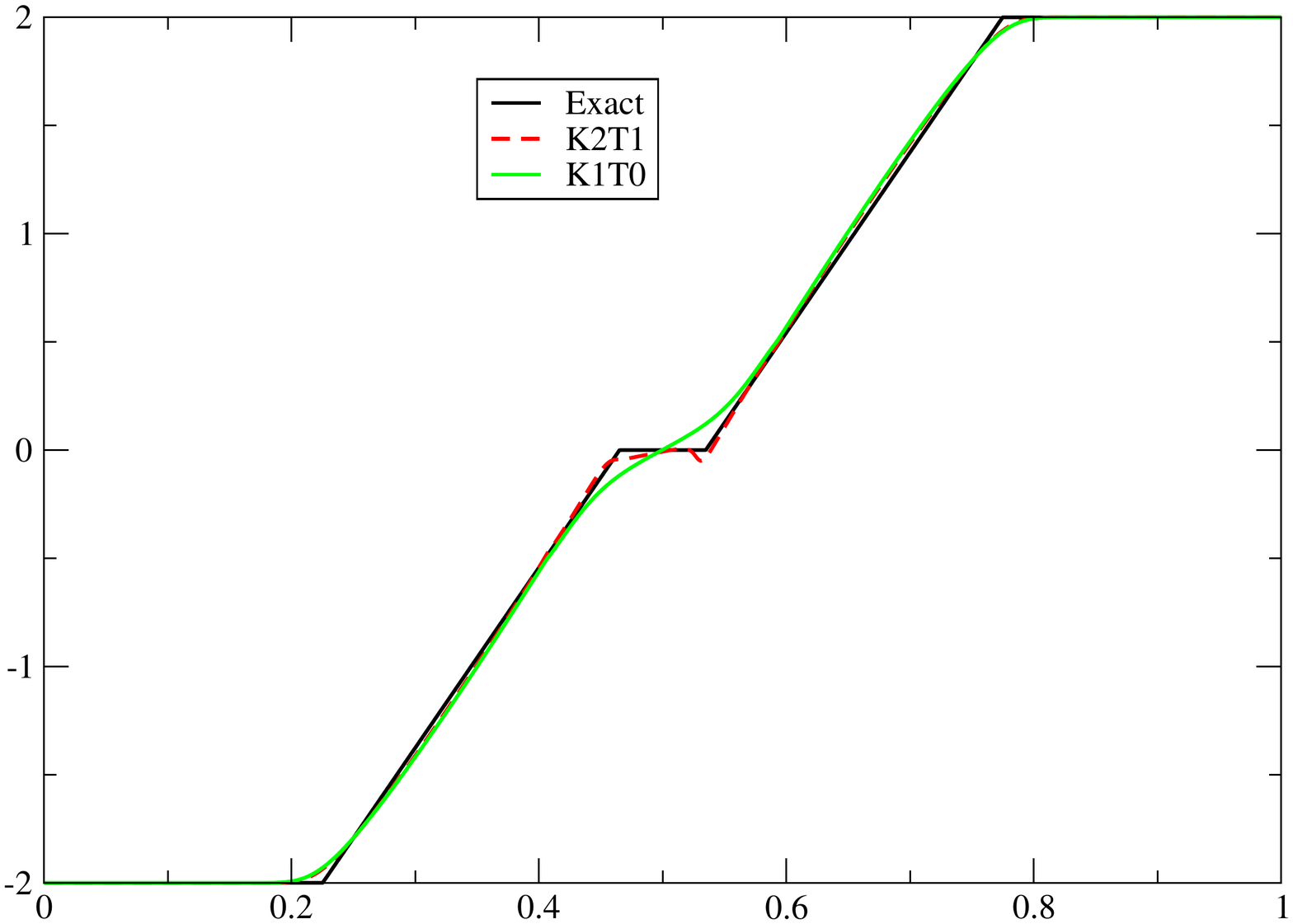}}
\end{center}

\caption{Solution of the $123$-problem for (a) density, (b) velocity and (c) pressure at time $T = 0.15$ with $CFL = 0.4$. In each figure, the exact solution and the solutions obtained with K2T1 and K1T0 are depicted.}
\label{123problem2}\end{figure}

\subsubsection{Severe test case}
The solution of the next test case consists of three strong discontinuities traveling to the right. The initial data consists of two constant states:
\begin{equation}\label{Collela1}
(\rho_0, u_0, p_0) = \left\{\begin{array}{ccc}
(5.99924,19.5975,460.894), & \text { if  }& 0.0\le x<0.8,
\\
(5.992420,-6.19633,46.0950), & \text { if  } & 0.8<x\le 1.0.
\end{array}\right. 
\end{equation}
This is one of the two test cases designed in \cite{MR2731357} which correspond to wave interaction in Collela and Woodward bast waves test case.
The exact and numerical solutions are found in the spatial domain $0 \le x \le 1$. The numerical solution is computed with $1000$ cells and the chosen Courant number coefficient is $0.1$. Boundary conditions are transmissive. The results for the density, velocity and pressure compared to the exact solution are shown in Figure \ref{interaction_blast} .

\begin{figure}
\begin{center}
\subfigure[]{\includegraphics[width=0.45\textwidth]{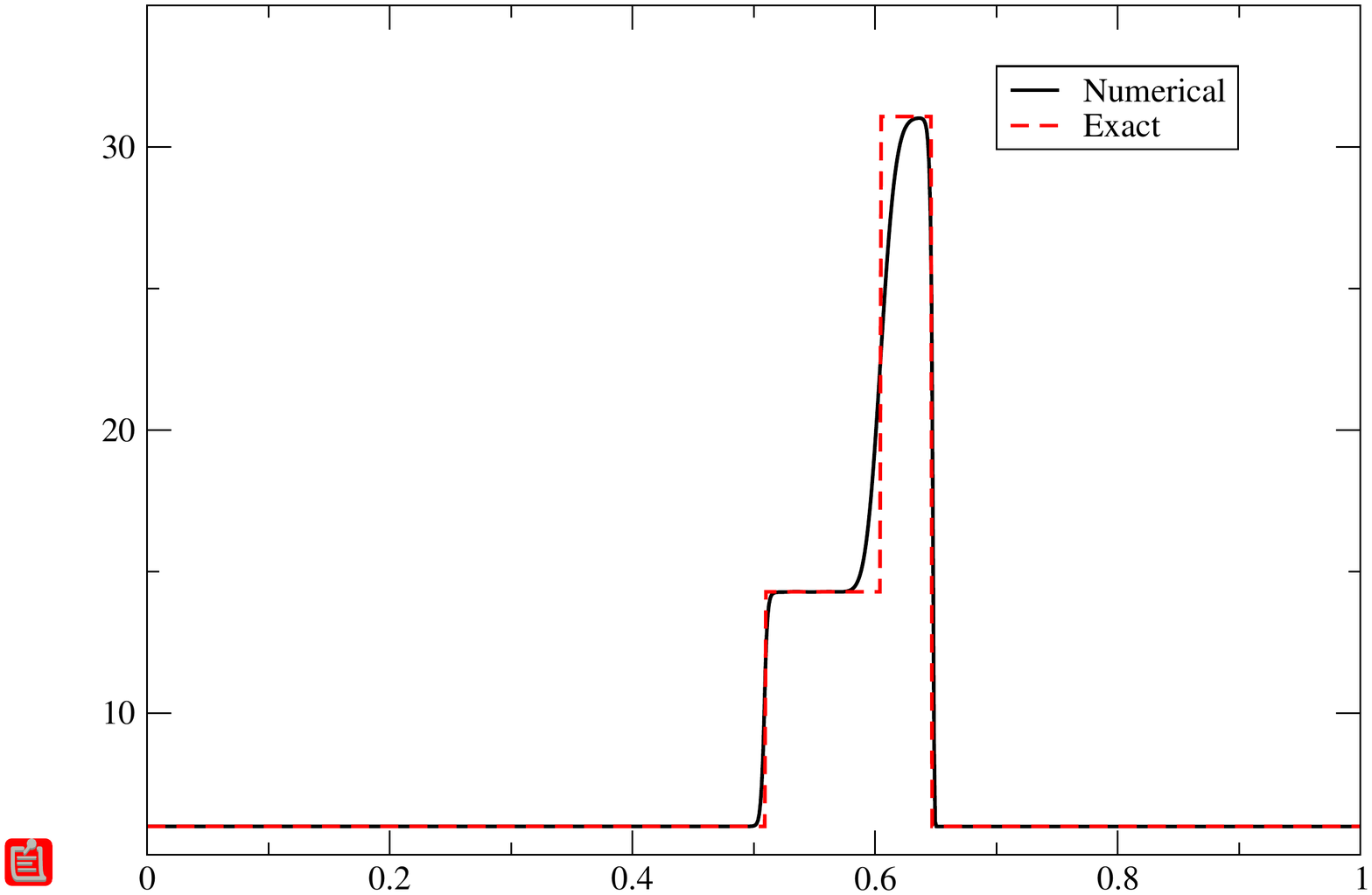}}
\subfigure[]{\includegraphics[width=0.45\textwidth]{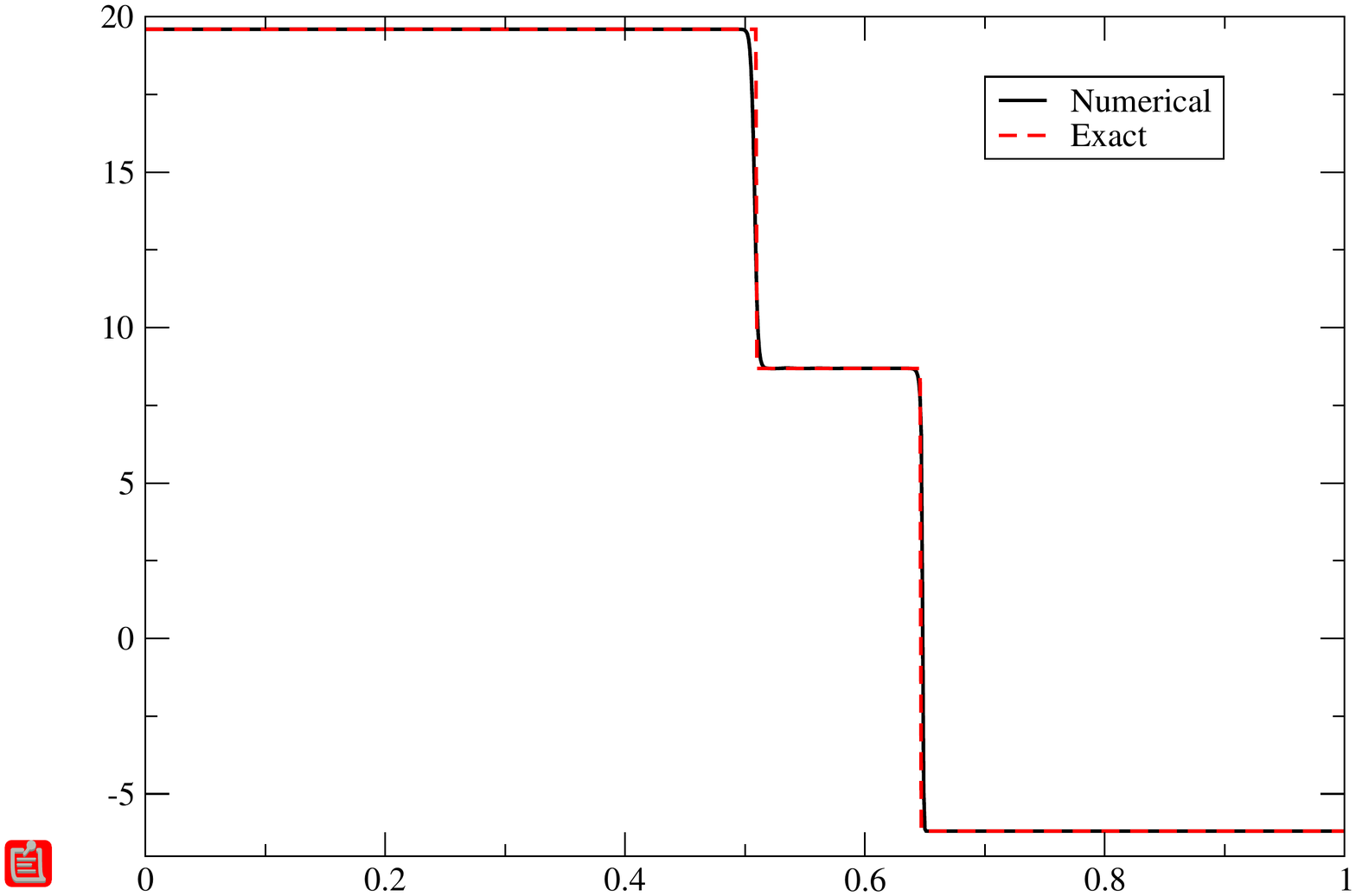}}
\subfigure[]{\includegraphics[width=0.45\textwidth]{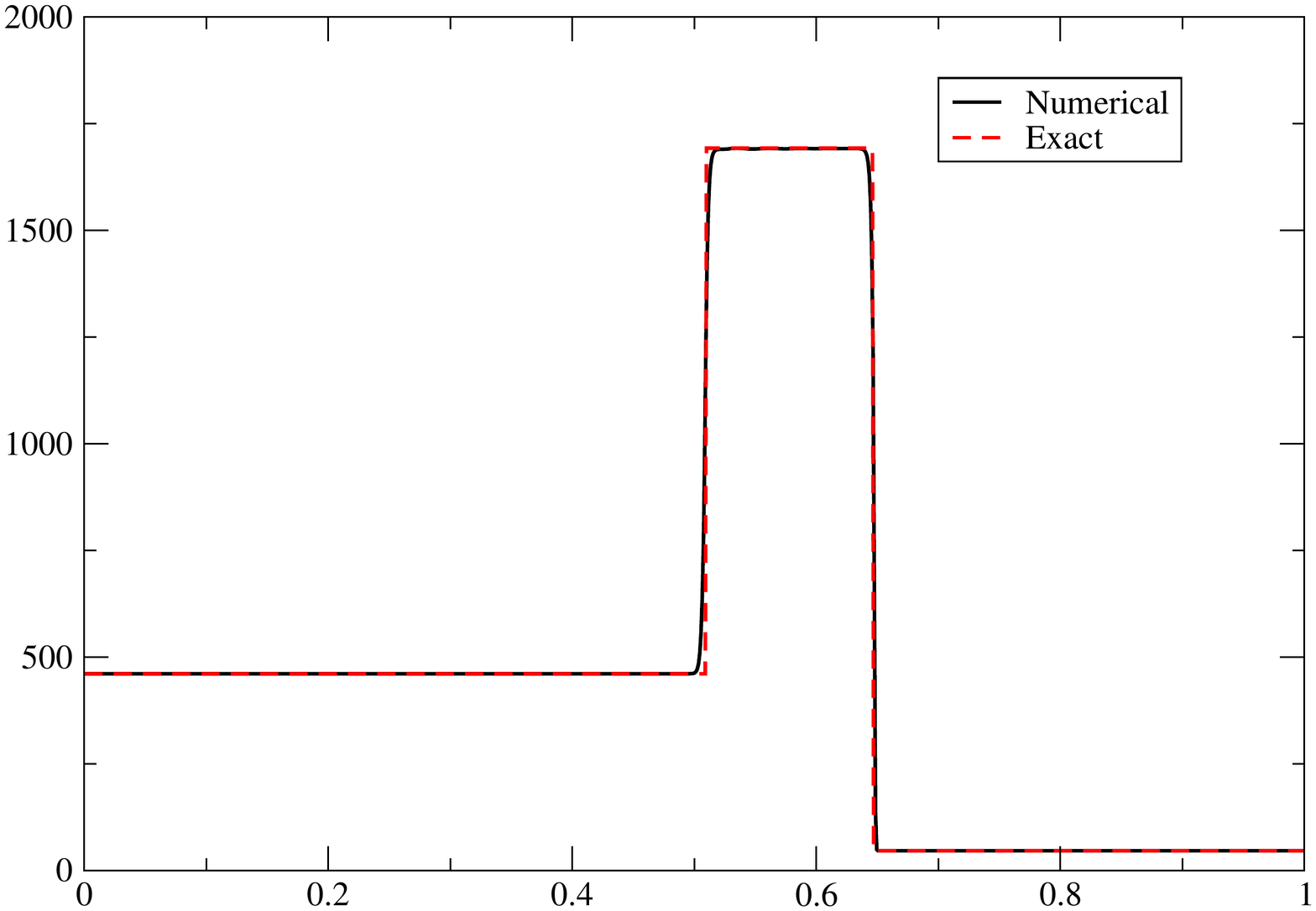}}
\end{center}

\caption{Exact and numerical solutions of the Collela and Woodward test case for (a) density, (b) velocity and (c) pressure at time $T = 0.012$ with $CFL = 0.1$ computed on a mesh with $1000$ points.}
\label{interaction_blast}\end{figure}

\subsubsection{2D test case}
{The scheme is a straightforward extension of the one dimensional one.}
In this test case we look at a 2D shock propagation in the domain $[0,1]\times [0,1]$. The initial data is given by:
\begin{equation}\label{Collela}
(\rho_0, u_0, v_0, p_0) = \left\{\begin{array}{ccc}
(1,0,0,1), & \text { if  }& ||x||<0.25,
\\
(0.125,0,0,0.1), & \text { if  } & ||x||\geq 0.25.
\end{array}\right. 
\end{equation}
The solutions for the density, velocity and pressure computed with the K1T0 first order-scheme as well as the used triangular mesh can be found in Figure \ref{2D}. It shows that the staggered scheme works in a stable way also in higher dimensions. By comparing the staggered scheme with the solution obtained by a conservative scheme in Figure \ref{2D2} we demonstrate that the shock front is correctly resolved.

{\color{black}We have also displayed in Figure \ref{fig:conservation} the relative error of conservation, i.e
$$
 \int_{\Omega} \bm_x^n\; d\bx, \quad \int_{\Omega} \bm_y\; d\bx, \text{ and } \dfrac{\int_\Omega E^n \; d\bx}{\int_\Omega E^0\; d\bx}-1.$$ There is no question on the density, since we use a dG scheme for that variable.
 We see that the errors are negligible. They also are independent of the choice of the quadrature formula  (not shown)
used to compute the parameters  needed in \eqref{numer:u} and \eqref{numer:e}.}
\begin{figure}
\begin{center}
\subfigure[]{\includegraphics[width=0.45\textwidth]{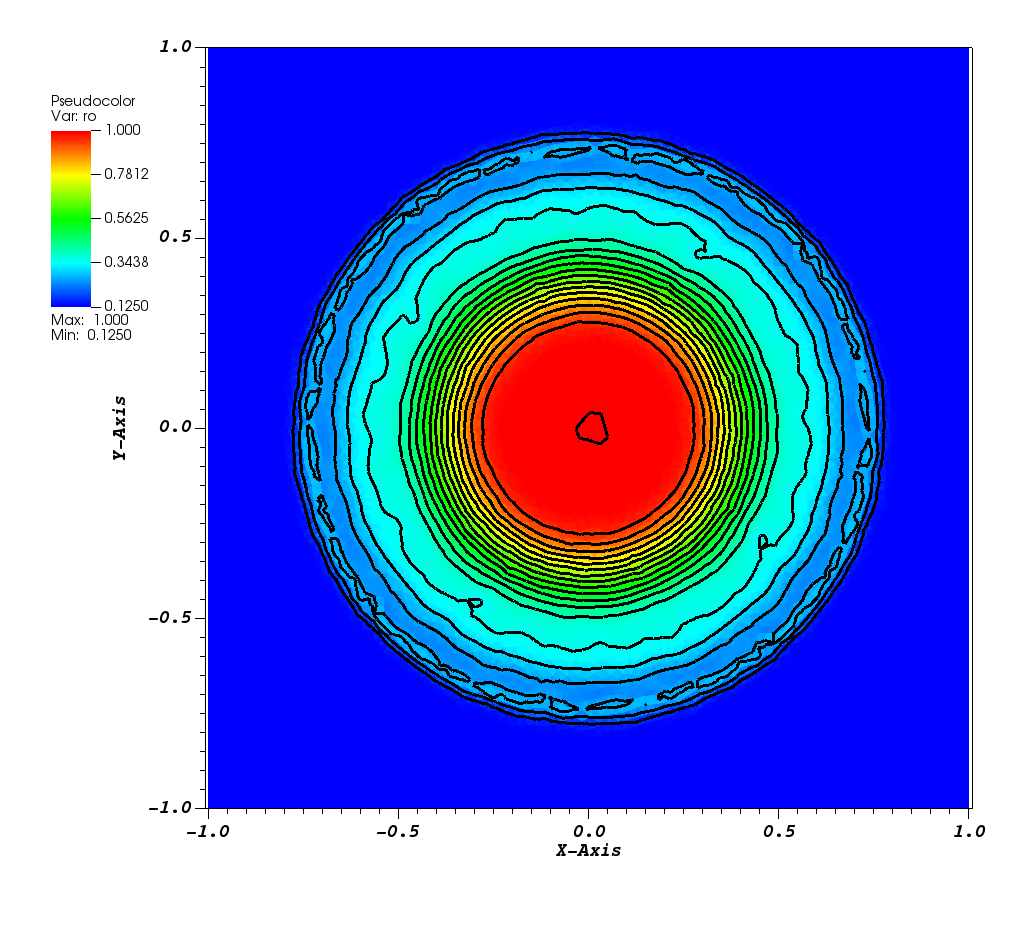}}
\subfigure[]{\includegraphics[width=0.45\textwidth]{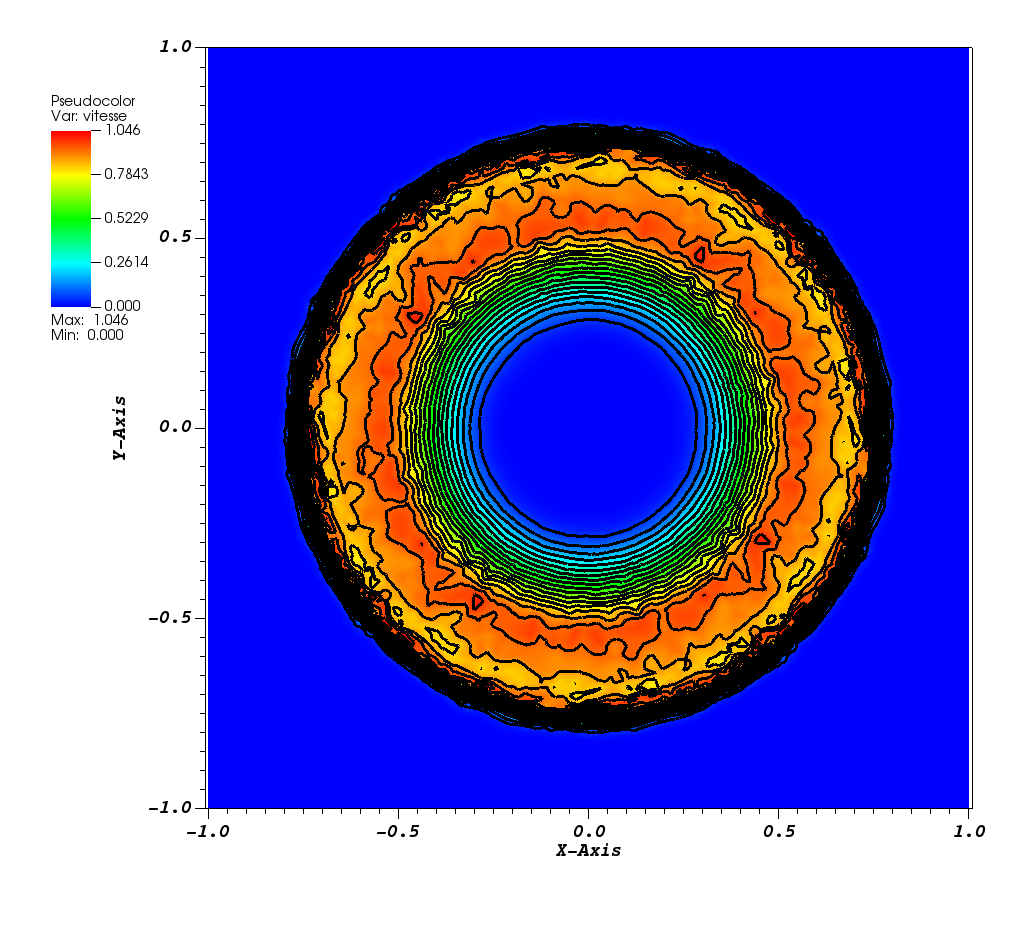}}

\subfigure[]{\includegraphics[width=0.45\textwidth]{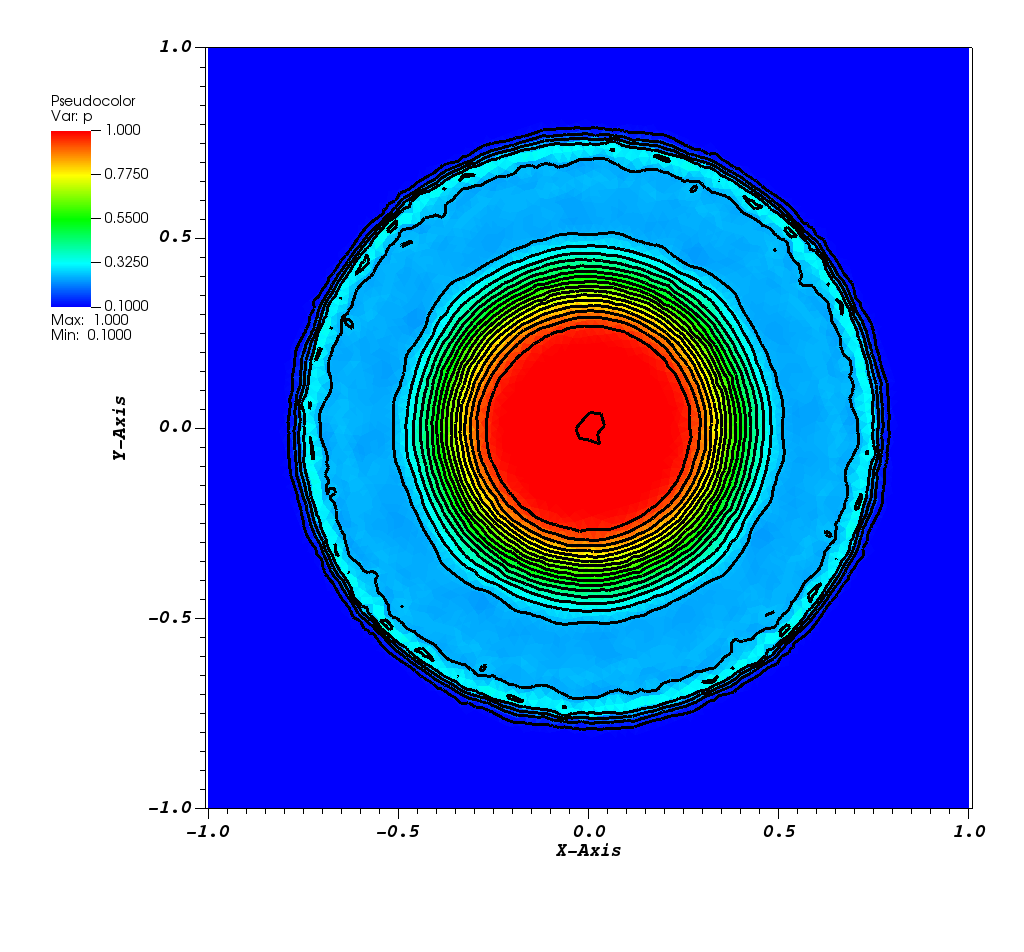}}
\subfigure[]{\includegraphics[width=0.45\textwidth]{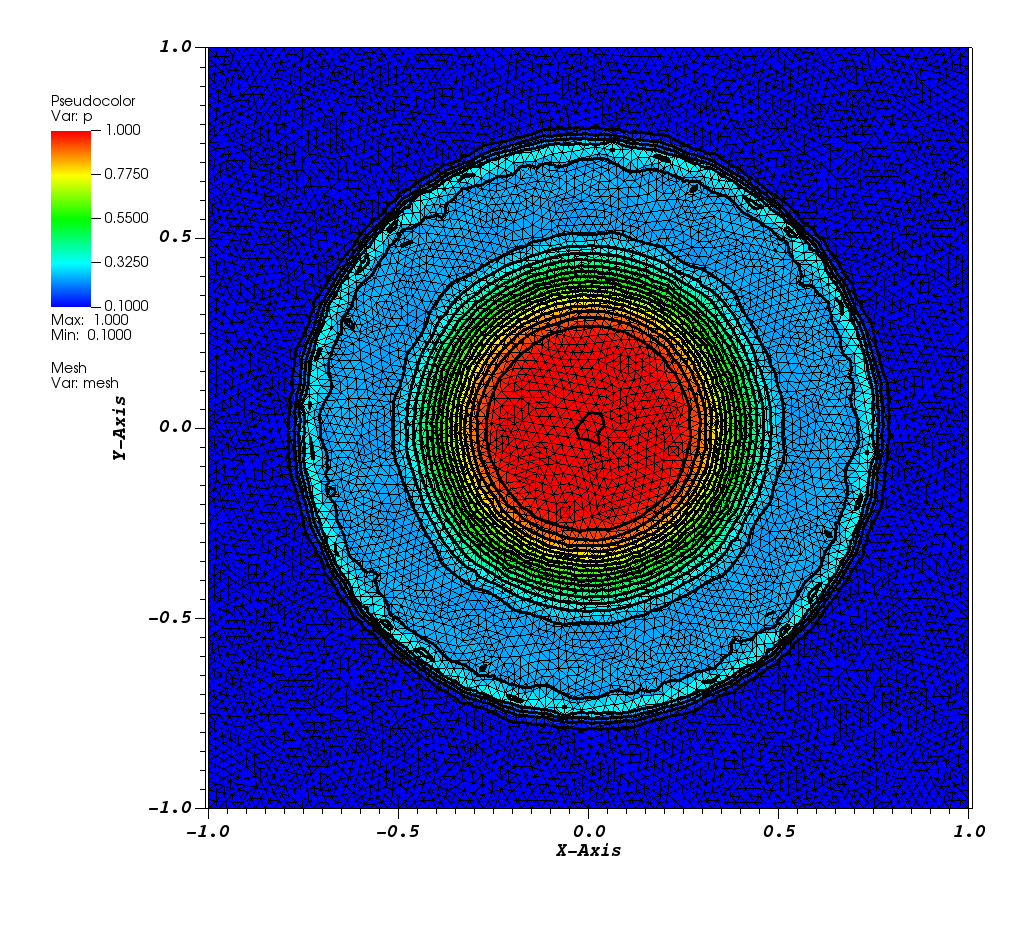}}
\end{center}
\caption{Numerical solutions of the 2D shock test case for (a) density, (b) vertical velocity and (c) pressure at time $T = 0.16$ computed on a mesh consisting of triangles which is depicted for the pressure in (d).}
\label{2D}\end{figure}

\begin{figure}
\begin{center}
{\includegraphics[width=0.45\textwidth]{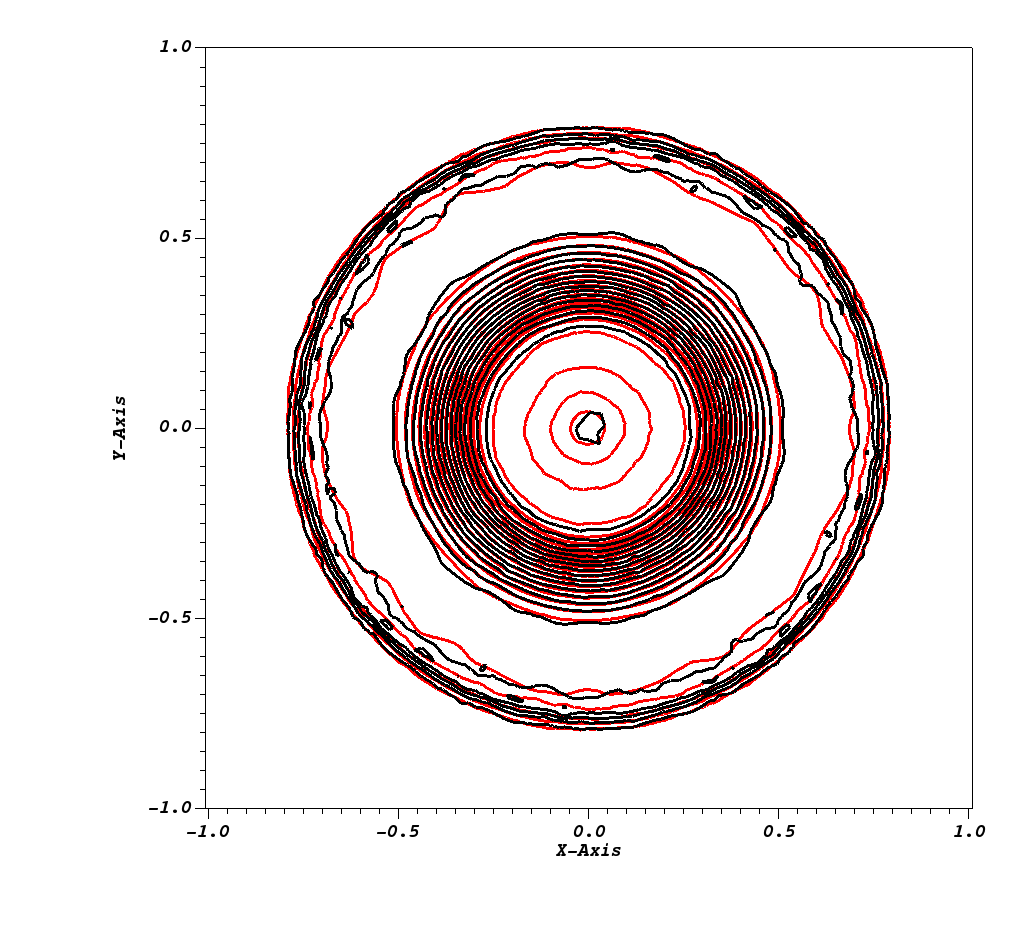}}
\end{center}
\caption{Comparison of the solutions of the 2D shock test case for the pressure obtained by a conservative scheme (red) and the staggered one (black) at time $T = 0.16$.}
\label{2D2}\end{figure}

\begin{figure}
\begin{center}
{\includegraphics[width=0.45\textwidth]{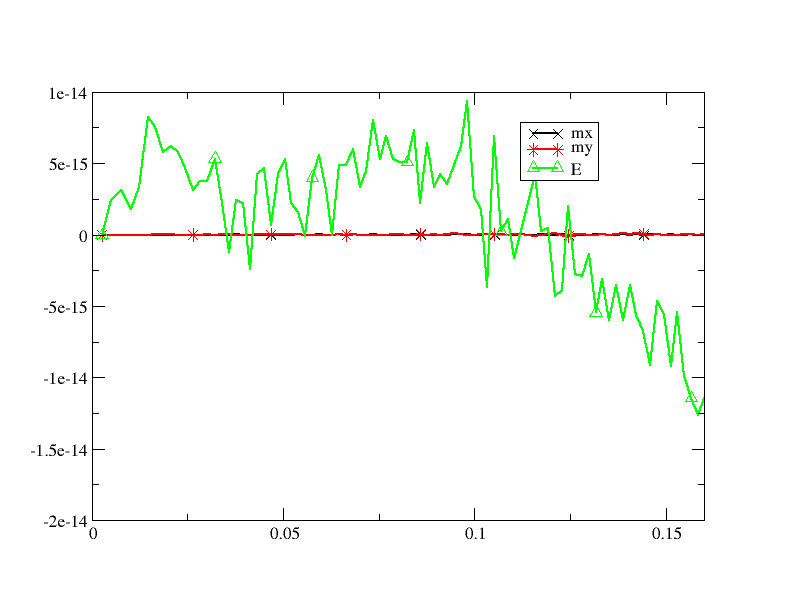}}
\end{center}
\caption{Conservation errors, in time.}
\label{fig:conservation}\end{figure}

{\subsection{Some remarks concerning the polynomial orders of the velocity and the thermodynamics parameters}
In all the simulations, we have made the choice of polynomials of degree $p$ for the thermodynamic parameter, and $p+1$ for the velocity one, but there is no justification except a rule of the thumb inspired  of what is done for incompressible flows. We have encountered stability problems in the case of equal degree polynomials. We discuss this in one dimension, the two dimensional case has not been explored. More precisely, when computing the quantities $p^\star$ for the velocity update, and the mass flux, several choices can be made in addition to the polynomial orders:
\begin{itemize}
\item We can take a centerred flux (arithmetic average between the left and right states), and $p^\star$ is also the arithmetic average. This will be the "centerred" choice.
\item We can compute the flux using the exact Godunov solver, as well as for $p^\star$. This will be the "exact" choice,
\item We can compute these quantities using the HLLC methodology. This will be the "HLLC" choice.
\end{itemize}
In the figure \ref{LBB} we report the results on the density for the smooth case of section \ref{sec:smooth} and time $t=0.025$. We see that all the combinations in polynomial orders with the exact, HLLC choices  are stable. The combination velocity with degree 1 and thermodynamics with degree 0 is also stable, while with linear  velocity and thermodynamics the scheme is not stable. This is why we have chosen $t=0.0025$ because soon after the code blows up. Only the polynomial degrees and the flux have  been changed, all of the other elements of the schemes remain the same. No limiting has been used, and the time accuracy is only first order.
\begin{figure}[h]
\begin{center}
\subfigure[]{\includegraphics[width=0.45\textwidth]{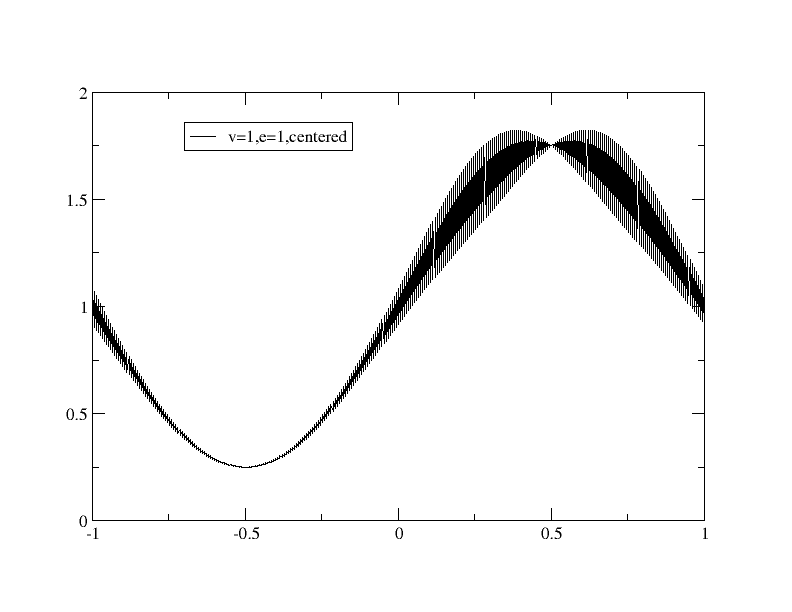}}
\subfigure[]{\includegraphics[width=0.45\textwidth]{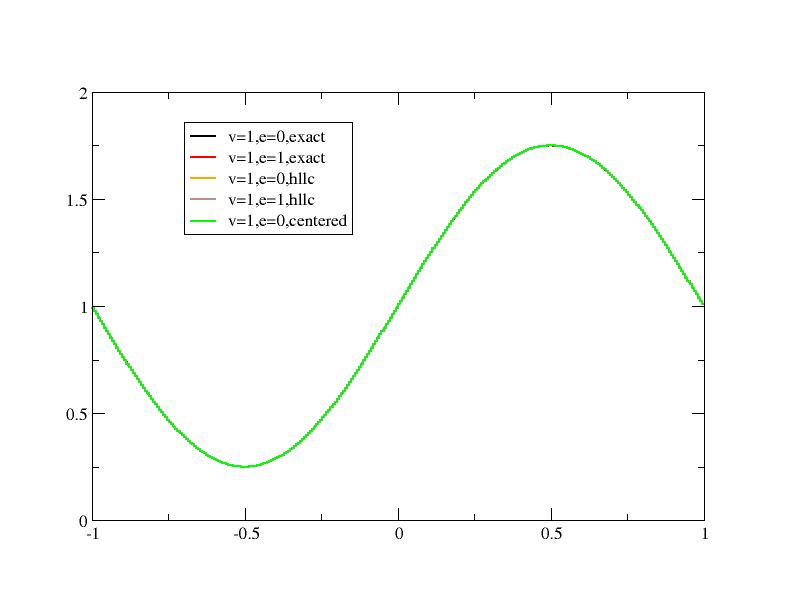}}
\end{center}
\caption{\label{LBB} (b) linear velocity, piecewise constant thermodynamics, (b) all the other cases.}
\end{figure}
 We have no mathematical explanation, only heuristic ones. We conjecture that for a centerred approximation, which is the closest with what is done for incompressible flows, we suffer of a kind of LBB stability problem. This stability problem is cured because of some "upwinding" mechanism with the exact and HLLC solver. 
 
  In all the other simulations, the velocity has been one degree more that the thermodynamics, and we take the exact or the HLLC solver, for security.}

%% file: lw.tex
{ \section{Proof of Proposition \ref{LxW}}\label{LxWthm}
We first show some estimates for scalar functions (the system case is identical), and then we use them to show Proposition \ref{LxW}. 
We start with some notations: $\R^d$ is subdivided into non-overlapping elements,
$$\R^d=\cup K$$ 
and the mesh is supposed to be conformal (because of the global continuity of the velocity). The parameter $h$ will be the maximum of the diameters of the $K$.  We assume that the partition is shape regular, i.e. the ratio between the inner and outer diameter of the elements is bounded from above and below. In $\R^d$,  we have a functional description of the density, the velocity and the energy: we call them $\rho_h$, $\bu_h$ and $e_h$ to refer they are defined from $\R^d=\cup K$.

 Let $T>0$ and let $0<t_1<\ldots <t_n<\ldots <t_N\leq T$ be a time discretisation of $[0,T]$.  We define $\Delta t_n=t_{n+1}-t_n$ and $\Delta t=\max\limits_n \Delta t_n$.  We are given the sequences $\{u_h^p\}^{p=0\ldots N}$, where $u_h^p$ belongs to $V^h$ or $W^h$ (see Section \ref{2.2}). They are defined from degrees of freedom that are again denoted by $\sigma$.  We can define  a function $u_{\Delta}$ by:
 $$\text{ if }(\bx,t)\in \Omega\times [t_{n}, t_{n+1}[, \text{ then } u _\Delta(\bx,t)=u_h^n(\bx).$$
 The set of these functions is denoted by $X _{\Delta}$ and is equipped with the $L^\infty$ and $L^2$ norms.

We then have the following lemma:
\begin{lemma}\label{weakBV}
Let $T>0$ , $\{t_n\}_{n=0, \ldots , N}$ an increasing  subdivision of $[0,T]$ and $\QQ$ a compact subset of $\R^d$. Let furthermore $(u _\Delta)_{h}$ denote a sequence of functions of $X _\Delta$ defined on $\R^d\times \R^+$. We assume that there exists $C\in \R$ independent of $\Delta$ and $\Delta t$, and $\bu\in L^2_{loc}(\Omega\times [0,T])$ such that
$$\sup\limits _\Delta\sup\limits_{\bx,t}\vert u _\Delta(\bx,t)\vert \leq C\quad \text{ and }\quad \lim\limits _{\Delta,  \Delta t\rightarrow 0}\vert u _\Delta-u\vert_{L^2(\Omega\times[0,T])}=0.$$
 Then, if $\overline{(u_h^n)}_K$ is the average of $u_h^n$ in $K$, we have
  \begin{equation}\label{appendix:1}\lim\limits_{h\rightarrow 0,\Delta t\rightarrow 0} \bigg (\sum\limits_{n=0}^N \Delta t _n\sum\limits_{K\subset \QQ} 
|K|\sum\limits_{\sigma
\in K} \norm{(u_h)_\sigma-\overline{(u_h)}_K)}\bigg )=0.\end{equation}
\end{lemma}
\begin{proof}
The proof is inspired from \cite{Kroner:96} and can be found in \cite{AbgrallRoe}.
\end{proof}

Now we have all the prerequisites for proving Proposition \ref{LxW}. We will perform the proof for the momentum since the proof for the energy is similar and can be done in a straightforward manner. We proceed the proof with several lemmas.
\begin{lemma}
Under the conditions of Proposition \ref{LxW}, for any $\varphi\in C_0^\infty(\R^d\times \R^+)$ we have 
\begin{equation*}
\begin{split}
\lim\limits_{\Delta t\rightarrow 0, \Delta\rightarrow 0
} \sum\limits_{n=0}^\infty\int_{\R^d} \varphi_h \big (\rho_h^{n+1}\bu_h^{n+1}&-\rho_h^n\bu_h^n\big ) \; d\bx\\
&=
-\int_{\R\times \R^+} \dpar{\varphi}{t} u \; dx dt+\int_\R \varphi(x,0) u_0\; dx dt,\end{split}\end{equation*}
where
$$\varphi_h(x, t_n)=\sum_K\varphi(x_K, t_k) 1_K\quad \text{ and }\quad \varphi_h(x,t)=\varphi(x, t_n) \text{ for } t\in [t_n,t_{n+1}[.$$
\end{lemma}
\begin{proof}
This is the classical lemma.
\end{proof}
\begin{proof}[Proof of Proposition \ref{LxW}]
We start from \eqref{master:1}
\begin{equation*}
\begin{split}
 \int_{\R^d}\psi(\bx,t) \; \big ( \rho^{n+1}\bu^{n+1}&-\rho^{n}\bu^{n}\big ) \; d\bx
\\
&+\Delta t_n \sum_K \psi_K^n \Bigg [ 
\sum_{\sigma_\mathcal{V}\in K} \omega^{\rho,n+1}_{\sigma_\mathcal{V}} \Phi_{\sigma_\mathcal{V},K}^\bu 
+\sum_{\sigma_{\mathcal{E}}\in K}\omega^{\bu,n,K}_{\sigma_\mathcal{V}}
\Phi^\rho_{\sigma_{\mathcal{E}},K}
 \Bigg ]\\
&\qquad \qquad \qquad+ \Delta t_n 
\sum_{K}\bigg ( {F}_K(\bu^{n})+\sum_{\sigma_\mathcal{V}\in K}D_{\sigma_\mathcal{V}}(\bu^{n})\bigg )=0,
\end{split}
\end{equation*}
and by using the assumptions of Proposition \ref{LxW} we obtain
\begin{equation*}
\begin{split}
 \int_{\R^d}\psi(\bx,t) \; \big ( \rho^{n+1}\bu^{n+1}&-\rho^{n}\bu^{n}\big ) \; d\bx+\sum_K \psi_K^n \Bigg [ \Delta t_n\int_{\partial K} \bbf^\bm \big (U^n\big )\cdot \bn \; d\gamma
 \Bigg ]\\
&\qquad \qquad \qquad+ \Delta t_n 
\sum_{K}\bigg ( {F}_K(\bu^{n})+\sum_{\sigma_\mathcal{V}\in K}D_{\sigma_\mathcal{V}}(\bu^{n})\bigg )=0.
\end{split}
\end{equation*}
From \eqref{master:2}, we see that 
\begin{equation*}
\begin{split}
F_K(\bu^{n})&=\sum_{\sigma_\mathcal{V}\in K }\big ( {\psi}_{\sigma_\mathcal{V}}^n-\psi_K^n\big )\omega^{\rho,n+1,K}_{\sigma_\mathcal{V}} \Phi_{\sigma_\mathcal{V},K}^\bu,\\
D_{\sigma_\mathcal{V}}(\bu^{n})&=\sum_{K', \sigma_\mathcal{V}\in K'}\bigg [\sum_{K, \sigma_\mathcal{V}\in K\cap K'}\omega^{\rho,n+1,K}_{\sigma_\mathcal{V}}\big (\psi_K^n-\psi_{\sigma_\mathcal{V}}^n\big ) \; \Phi_{\sigma_\mathcal{V}, K'}^{\bu}\bigg ],
\end{split}
\end{equation*}
so that, since $\psi_K^n-\psi_{\sigma_\mathcal{V}}^n=O(h)$, using the estimates of Lemma \ref{weakBV}, we have
$$\lim\limits_{\Delta t_n, h\rightarrow 0}\Delta t_n\sum_{K}{F}_K(\bu^{n})=0 \quad\text{and}\quad \lim\limits_{\Delta t_n, h\rightarrow 0}\Delta t_n\sum_{\sigma_\mathcal{V}\in K}D_{\sigma_\mathcal{V}}(\bu^{n})=0$$
because the mesh is shape regular and $\nicefrac{\Delta t_n}{h}$ is bounded. Last, using the same technique as in \cite{AbgrallRoe}, and due again to the Lemma \ref{weakBV}, we see that
$$\lim\limits_{\Delta t_n, h\rightarrow 0}\sum_K \psi_K^n 
 \Delta t_n\int_{\partial K} \bbf^\bm \big (U^n\big )\cdot \bn \; d\gamma=\int_{\R^+}\int_{\R^d}\nabla_\bx \psi(\bx,t)\bbf^m(U)\; d\bx.$$
The convergence result for the energy is done with exactly the same method which then finishes the proof of Proposition \ref{LxW}.\end{proof}
}